\documentclass[reqno]{amsart}

\usepackage{a4wide, mathrsfs, amssymb, array, esint}

\usepackage[colorlinks=true]{hyperref}

\numberwithin{equation}{section}

\theoremstyle{plain}
\newtheorem{theorem}{Theorem}[section]
\newtheorem*{theorem*}{Theorem}
\newtheorem{corollary}[theorem]{Corollary}
\newtheorem{proposition}[theorem]{Proposition}
\newtheorem{lemma}[theorem]{Lemma}

\theoremstyle{definition}
\newtheorem{definition}[theorem]{Definition}

\theoremstyle{remark}
\newtheorem{remark}[theorem]{Remark}

\newcommand\eps{\ensuremath{\varepsilon}}

\newcommand\e{\ensuremath{\mathbf{e}}}
\newcommand\norm[1]{\ensuremath{\left\Vert {#1} \right\Vert}}
\newcommand\ipe[2]{\ensuremath{\langle {#1}, {#2} \rangle_{\e}}}
\newcommand\norme[1]{\ensuremath{\big| #1 \big|_{\e}}}
\newcommand\loc{\ensuremath{\mathrm{loc}}}

\newcommand\Wloc[1]{\ensuremath{W^{#1}_{\loc}(M)}}

\newcommand\Lp[1]{\ensuremath{L^{#1}_{\loc}(M)}}
\newcommand\x{\ensuremath{\mathbf{x}}}
\newcommand\adm{\ensuremath{m_{\mathrm{ADM}}}}

\newcommand\sgmin{\ensuremath{( s_{\g} )_-}}
\newcommand\sepsmin{\ensuremath{( s_{\g_{\eps}} )_-}}
\newcommand\shmin{\ensuremath{( s_{\h} )_-}}

\newcommand\g{\ensuremath{\mathbf{g}}}
\newcommand\h{\ensuremath{\mathbf{h}}}
\newcommand\geps{\ensuremath{\g_{\eps}}}
\newcommand\ghat{\ensuremath{\widehat{\mathbf{g}}}}
\newcommand\hhat{\ensuremath{\widehat{\mathbf{h}}}}
\newcommand\T{\ensuremath{\mathbf{T}}}
\renewcommand\H{H\"{o}lder}

\newcommand\R{\ensuremath{\mathbb{R}}}

\newcommand\la{\langle}
\newcommand\ra{\rangle}

\newcommand\bel[1]{\begin{equation}\label{#1}}
\newcommand\ee{\end{equation}}

\newcommand\Dprime{\ensuremath{\mathscr{D}'}}
\newcommand\D{\ensuremath{\mathscr{D}}}

\DeclareMathOperator{\supp}{supp}

\begin{document}
\title{A positive mass theorem for low-regularity Riemannian metrics}
\date{\today}

\author[J.D.E.\ Grant]{James~D.E.\ Grant}
\address{\href{http://www2.surrey.ac.uk/maths/index.htm}{Department of Mathematics} \\ Faculty of Engineering and Physical Sciences \\ \href{http://www2.surrey.ac.uk/}{University of Surrey} \\ Guildford \\ GU2 7XH \\ United Kingdom}
\email{\href{mailto:j.grant@surrey.ac.uk}{j.grant@surrey.ac.uk}}
\urladdr{\href{http://jdegrant.wordpress.com}{http://jdegrant.wordpress.com}}

\author[N.\ Tassotti]{Nathalie Tassotti}
\address{\href{http://www.mat.univie.ac.at/home.php}{Fakult\"{a}t f\"{u}r Mathematik} \\
\href{http://www.univie.ac.at}{Universit\"{a}t Wien} \\ Oskar-Morgenstern-Platz 1 \\ 1090 Vienna \\ Austria}
\email{\href{nathalie.tassotti@univie.ac.at}{nathalie.tassotti@univie.ac.at}}
\urladdr{\href{http://nathalietassotti.wordpress.com}{http://nathalietassotti.wordpress.com}}

\begin{abstract}
We show that the positive mass theorem holds for continuous Riemannian metrics that lie in the Sobolev space $W^{2, n/2}_{\loc}$ for manifolds of dimension less than or equal to $7$ or spin-manifolds of any dimension. More generally, we give a (negative) lower bound on the ADM mass of metrics for which the scalar curvature fails to be non-negative, where the negative part has compact support and sufficiently small $L^{n/2}$ norm. We show that a Riemannian metric in $W^{2, p}_{\loc}$ for some $p > \frac{n}{2}$ with non-negative scalar curvature in the distributional sense can be approximated locally uniformly by smooth metrics with non-negative scalar curvature. For continuous metrics in $W^{2, n/2}_{\loc}$, there exist smooth approximating metrics with non-negative scalar curvature that converge in $L^p_{\loc}$ for all $p < \infty$.
\end{abstract}

\thanks{We are grateful to Prof.~M.\ Kunzinger for a discussion that clarified the proof of Lemma~\ref{lemma:smooth}, and for checking some parts of this paper.
The work of NT was funded by a Forschungsstipendium from the University of Vienna, Grant P23719-N16 of the \href{http://www.fwf.ac.at/}{FWF} and a DOC fellowship from the~\href{http://http://www.oeaw.ac.at/}{Austrian Academy of Sciences}. JG is grateful to \href{http://www.sjc.ox.ac.uk/}{St.~John's College}, \href{http://www.ox.ac.uk/}{University of Oxford} for a Visiting Scholarship during which this paper was completed. He is also grateful to the \href{https://www.maths.ox.ac.uk/}{Mathematical Institute} in Oxford for their hospitality.}

\maketitle
\thispagestyle{empty}

\section{Introduction}
\label{sec:Introduction}

An important result in the theory of Riemannian metrics with non-negative scalar curvature is the positive mass theorem. In this paper, we show that the positive mass theorem continues to hold for a class of low-regularity metrics, and attempt to distil from our proof the important geometrical structures that are required for validity of our results. Our work may, therefore, be considered as a contribution to (or, at least, was motivated by) the development of a ``synthetic'' approach to non-negative scalar curvature metrics, in the sense that Alexandrov spaces with curvature bounded above/below are a natural generalisation of Riemannian manifolds with upper/lower bounds on sectional curvature, and the metric measure spaces of Lott--Villani and Sturm generalise the notions of Riemannian metrics with Ricci curvature bounded below.

We first recall the standard formulation of the positive mass theorem for Riemannian metrics. A Riemannian manifold $(M, \g)$ of dimension $n \ge 3$ is \emph{asymptotically flat\/} if there exists a compact subset $K$ of $M$ such that the end $M \setminus K$ is diffeomorphic to $\R^n \setminus \overline{B(0, r_0)}$ for some $r_0 > 0$, i.e. there exists a diffeomorphism $\Phi \colon \R^n \setminus \overline{B(0, r_0)} \to M \setminus K$. (For simplicity, we will assume throughout that our manifold has only one end. The generalisation to multiple ends is straightforward.) In terms of this coordinate chart, it is required that the pull-back of the metric $\g$ have the fall-off properties:
\begin{subequations}
\begin{align}
\left( \Phi^* \g \right)_{ij} - \delta_{ij} &= O \left( \frac{1}{|\x|^q} \right), 
\\
\partial_k \left( \Phi^* \g \right)_{ij} &= O \left( \frac{1}{|\x|^{q+1}} \right), 
\\
\partial_k \partial_l \left( \Phi^* \g \right)_{ij} &= O \left( \frac{1}{|\x|^{q+2}} \right), 
\end{align}\label{asymptotics}\end{subequations}
as $|\x| \to \infty$, for some $q > \frac{n-2}{2}$. We also demand that the scalar curvature of $\g$, $s_{\g}$, satisfies the integrability condition
\begin{equation}
s_{\g} \in L^1(M).
\tag{1.1d}
\end{equation}
We define the \emph{ADM mass\/} of $(M, \g)$ to be
\[
\adm(\g) := \frac{1}{2 (n-1) \omega_{n-1}} \lim_{r \to \infty} \int_{S_r} \left( g_{ij, i} - g_{ii, j} \right) \nu_i \, d\mu_0,
\]
where $S_r := \Phi \left( \{ \x \in \R^n \mid |\x| = r \} \right)$ denotes the coordinate sphere in the end $M \setminus K$, $\nu$ and $d\mu_0$ denote the outward-pointing unit normal to $S_r$ and induced Euclidean volume element on $S_r$, respectively, and $\omega_{n-1}$ denotes the area of the unit $(n-1)$ sphere in $\R^n$. Under the conditions imposed in~\eqref{asymptotics}, the mass is a well-defined quantity and independent of the asymptotically flat coordinate chart $\Phi$~\cite{Bartnik}.

If the manifold $M$ and the Riemannian metric $\g$ are smooth, and additionally $M$ is either a manifold of dimension $n \le 7$~\cite{SY} or a spin manifold of any dimension $n \ge 3$~\cite{Witten}, one has the following result.

\begin{theorem*}[Positive mass theorem]
Let $(M, \g)$ be a complete, asymptotically flat Riemannian manifold, where the metric $\g$ has non-negative scalar curvature. Then $\adm(\g) \ge 0$.
\end{theorem*}

In this paper, we show that the positive mass theorem remains valid for continuous metrics $\g$ that lie in the local Sobolev space $W^{2, n/2}_{\loc}(M)$. We assume that $M$ is either a spin manifold, or that $n \le 7$ in order that the classical positive mass theorem holds. For simplicity, we assume that the metric $\g$ is smooth outside of the compact set $K$ and satisfies the asymptotic conditions~\eqref{asymptotics}, although it is also straightforward to generalise our results to the case where the metric lies in an appropriate weighted space outside of $K$ as in~\cite{Bartnik}. Since $\g \in C^0(M) \cap W^{2, n/2}_{\loc}(M)$, the scalar curvature of $\g$ lies in $L^{n/2}_{\loc}(M)$ and, therefore, is well-defined as a distribution on $M$. One of our main results is then the following:%
\footnote{This result was announced in~\cite{pmt1}, with a sketch of the main points of the proof.}

\begin{theorem*}
Let $(M, \g)$ be a complete, asymptotically flat Riemannian manifold, with $\g \in C^0(M) \cap W^{2, n/2}_{\loc}(M)$ and $s_{\g}$ non-negative in the distributional sense (i.e. $\la s_{\g}, \varphi \ra \ge 0, \forall \varphi \in \D(M)$). Then the ADM mass $\adm(\g)$ is non-negative.
\end{theorem*}

In fact, our methods allow us to prove a stronger result, for metrics $\g$ for which the scalar curvature is not constrained to be non-negative. Denoting the Sobolev constant of the metric $\g$ by $c_1[\g]$, and the Riemannian measure of a measurable set $E \subseteq M$ by $\mu_{\g}(E)$, we show the following.

\begin{theorem*}
Let $(M, \g)$ be a complete, asymptotically flat Riemannian manifold, with $\g \in C^0(M) \cap W^{2, n/2}_{\loc}(M)$ with the properties that $\sgmin$ is of compact support and that
\[
c_1[\g] \, \norm{\sgmin}_{L^{n/2}(M, \g)} < 4 \, \frac{n-1}{n-2}.
\]
Then the mass of the metric $\g$ satisfies
\[
\adm(\g)
\ge - \frac{1}{2 (n-1) \omega_{n-1}} \frac{\norm{\sgmin}_{L^{n/2}(M, \g)}}{\left( 1 - \frac{n-2}{4 (n-1)} c_1[\g] \, \norm{\sgmin}_{L^{n/2}(M, \g)} \right)^2} \, \mu_{\g}(\supp \sgmin)^{2/n^*}.
\]
\end{theorem*}
Therefore, as long as the negative part of the scalar curvature of $\g$ is of compact support and has sufficiently small $L^{n/2}$ norm, we have a (negative) lower bound on mass. To our knowledge, this is the first result of this type.%

\smallskip
Our approach to proving the positive mass theorem is to construct appropriate smooth approximations to the metric $\g$, and is based on a modification of the approach of Miao~\cite{Miao} (see, also,~\cite{McS, Lee}). In~\cite{Miao}, metrics were considered where the singular behaviour was localised on a hypersurface. By an ingenious mollification technique, he smoothed out the metric on a neighbourhood of this hypersurface in a controlled way, while leaving the metric unchanged on the rest of the manifold. The metrics that we consider can be non-smooth on an arbitrary compact subset $K \subset M$, so we cannot directly apply this technique. Nevertheless, we proceed by smoothing the metric on the compact set $K$ but leaving it unchanged on $M \setminus K$. This implies that the smooth approximating metrics $\g_{\eps}$ have the same asymptotic behaviour, and hence the same mass, as the metric $\g$. The metrics $\g_{\eps}$ generally no longer have non-negative scalar curvature, but we may perform a conformal transformation to give a new family of smooth metrics $\ghat_{\eps}$ that are both asymptotically flat and have non-negative scalar curvature. The classical positive mass theorem above then implies that $\adm(\ghat_{\eps}) \ge 0$. Using elliptic estimates, we show that $\adm(\ghat_{\eps}) \to \adm(\g)$ as $\eps \to 0$, thereby implying that $\adm(\g) \ge 0$.

As a significant by-product of our approach, we establish results concerning the approximation of rough metrics with non-negative scalar curvature by smooth metrics with non-negative scalar curvature. In the case where the metric $\g$ is assumed to lie in the Sobolev space $W^{2, p}_{\loc}$ for some $p > \frac{n}{2}$ (and therefore, by the Sobolev embedding theorem, is continuous) we have the following approximation theorem:
\begin{theorem*}
Let $\g$ be a Riemannian metric on an open set $\Omega$ of regularity $W^{2, p}_{\loc}(\Omega)$, $p > \frac{n}{2}$ with non-negative scalar curvature in the distributional sense. Then there exists a family of smooth, Riemannian metrics $\{ \ghat_{\eps} \mid \eps > 0 \}$ with non-negative scalar curvature such that $\ghat_{\eps}$ converge locally uniformly to $\g$ as $\eps \to 0$.
\end{theorem*}
The elliptic estimates that we require to prove this result break down for metrics $\g \in C^0(M) \cap W^{2, n/2}_{\loc}(M)$. However, we can show the following: \begin{theorem*}
Let $\g$ be a Riemannian metric on an open set $\Omega$ of regularity $C^0(\Omega) \cap W^{2, n/2}_{\loc}(\Omega)$ with non-negative scalar curvature in the distributional sense. Then there exists a family of smooth, Riemannian metrics $\{ \ghat_{\eps} \mid \eps > 0 \}$ with non-negative scalar curvature such that $\ghat_{\eps}$ converge $\g$ in $L^p_{\loc}(\Omega)$ as $\eps \to 0$, for all $p < \infty$.
\end{theorem*}
In general, we do not expect that it will be possible to locally uniformly approximate continuous metrics in $W^{2, n/2}_{\loc}$ with non-negative scalar curvature by smooth metrics with non-negative scalar curvature. Our results suggest an underlying ``bubbling off'' or non-compactness phenomenon. In light of the critical Sobolev embedding of $W^{2, n/2}$, we expect, however, that the approximating metrics that we construct should converge to $\g$ in appropriate BMO or Orlicz spaces~\cite{Trudinger:Orlicz}. We have, however, not investigated this possibility.

\smallskip
Returning to the positive mass theorem, in addition to the work~\cite{Miao, McS, Lee} where metrics with particular types of singularities were considered, there is previous working aimed at lowering the regularity requirements necessary for the validity of the Witten proof of the positive mass theorem. In~\cite{Bartnik, BCPreprint} and, very recently,~\cite{LL} it was shown that the Witten proof remains valid for Riemannian metrics in various Sobolev spaces. In particular,~\cite{LL} gives a proof for metrics in $C^0(M) \cap W^{1, n}_{\loc}(M)$ on spin manifolds.%
\footnote{JG is grateful to Dan Lee for discussing the results of~\cite{LL} prior to the posting of their paper.}

As far as we are aware, there is no previous work aimed at lowering the regularity of the metric required to carry through the Schoen--Yau proof of the positive mass theorem. Our work may be interpreted as a step in this direction, although it would be of interest to determine the precise regularity conditions required to carry through the Schoen--Yau proof of the positive mass theorem directly. In light of our wish to find a synthetic approach to spaces with non-negative scalar curvature, it will become clear that the Sobolev inequality~\ref{lemma:SY} (Lemma~3.1 in~\cite{SY}) plays a crucial role in our constructions. In line with the integral conditions on the negative part of the scalar curvature that arise in our constructions, it would appear that the general geometrical framework of ``almost smooth metric-measure spaces'' with a Sobolev inequality developed in~\cite{ACM} to study the Yamabe problem on singular Riemannian spaces may also be worthy of investigation from the point of view of the positive mass theorem.

\smallskip

For the general case of non-spin manifolds, it is difficult to see how the results of the present paper can be improved upon by the methods that we have used. Both continuity and the $W^{2, n/2}_{\loc}$ nature of the metric have played an important part. At this moment, there are quite distinct differences in the regularity required to prove the positive mass theorem on spin manifolds and non-spin manifolds, particularly concerning rigidity theorems. It would certainly be of interest to determine whether there is any deep relation between the existence of spin structures on manifolds and the regularity requirements for validity of the positive mass theorem.

\medskip

This paper is organised as follows. In Sections~\ref{sec:rough} and~\ref{sec:nnsc}, we introduce the class of metrics that we wish to study, and establish some preliminary analytical results. In Section~\ref{sec:approx}, we establish the existence of appropriate smooth approximations, $\g_{\eps}$ to our rough metric $\g$ that agree with $\g$ outside of a compact set. We then prove convergence results for related geometrical quantities, such as Sobolev constants and the negative part of the scalar curvature of the metrics $\g_{\eps}$, that will be required later. In Section~\ref{sec:elliptic}, we consider the conformal rescaling of the metrics $\g_{\eps}$ to yield smooth metrics $\ghat_{\eps}$ that have non-negative scalar curvature. We develop some detailed elliptic estimates for the conformal factors required, sharpening some of the estimates developed in~\cite{SY} and~\cite{Miao}. With these estimates in hand, in Section~\ref{sec:mass}, we complete the proof of the positive mass theorem for our metrics. In addition, we establish that, if we allow the metric $\g$ to have negative scalar curvature on a compact set, and that the $L^{n/2}$ norm of $\sgmin$ is sufficiently small, then we can put a (negative) lower bound on the ADM mass $\adm(\g)$.

Section~\ref{sec:approximation} of the paper is largely independent of the positive mass theorem, and studies the approximation of metrics of regularity $W^{2, p}_{\loc}$ for some $p > \frac{n}{2}$. In Theorems~\ref{thm:lowervest} and~\ref{thm:uppervest}, we establish estimates that control the $L^{\infty}$ norm of the conformal transformations required to construct smooth approximations to the metric $\g$ by non-negative scalar curvature metrics $\ghat_{\eps}$. This requires a rather detailed Moser iteration argument. From these estimates, we establish that the conformal factors in the construction of the $\ghat_{\eps}$ converge locally uniformly to $1$, and therefore that a $W^{2, p}_{\loc}$ metric with non-negative scalar curvature in the distributional sense can be approximated locally uniformly by smooth metrics with non-negative scalar curvature. These iteration arguments break down, however, for metrics that lie in $C^0(M) \cap W^{2, n/2}_{\loc}(M)$. In particular, for given $\eps > 0$, we can only perform a finite number of Moser iterations, to give an $L^p$ estimate for the conformal factor, where $p \to \infty$ as $\eps \to 0$. As such, we establish that such metrics can be approximated by smooth metrics with non-negative scalar curvature that converge to $\g$ in $L^p_{\loc}(M)$ for all $p < \infty$. We argue, however, that one cannot expect $L^{\infty}$ convergence.

We close with some remarks concerning the rigidity part of positive mass theorem.

\section{Rough metrics}
\label{sec:rough}

Let $M$ be a smooth manifold of dimension $n \ge 3$.%
\footnote{For all of our considerations, $M$ being $C^3$ would be sufficient.}
We assume that $M$ is equipped with a Riemannian metric $\g$ and is asymptotically flat, in the sense that there
exists a compact subset $K \subset M$ such that the end $N := M \setminus K$ is diffeomorphic to $\R^n \setminus \overline{B(0, r_0)}$ for some $r_0 > 0$. We assume that the metric $\g$ is smooth on $M \setminus K$, and that, on each end, $\g$ satisfies the asymptotic conditions~\eqref{asymptotics}. In addition, we assume that the classical positive mass theorem is valid on $M$, so either $M$ is a spin-manifold of arbitrary dimension~\cite{Witten} or $n \le 7$~\cite{SY}.

In order to formulate the regularity conditions that we require of $\g$, it is convenient to introduce a smooth background Riemannian metric on $M$.%
\footnote{It is not necessary to do so, however. We could, equivalently, work in local charts.}
Let {\e} be an arbitrary smooth Riemannian metric on $M$ that agrees with $\g$ on the set $M \setminus K$.%
\footnote{This can be done without loss of generality, enlarging the set $K$ if necessary.}
We denote the inner products and norms induced by {\e} on tensor bundles on $M$ by {\ipe{\cdot}{\cdot}} and {\norme{\cdot}}, respectively. We denote by $d\mu_{\e}$ the Riemannian measure defined by the metric $\e$.

Given a tensor field, $\T$, on $M$, and $p \in [1, \infty]$, we define the $L^p$ norm of $\T$ on a set $U \subseteq M$ with respect to $\e$ to be
\[
\norm{\T}_{L^p(U)} := \left( \int_U \norm{\T}_{\e}^p \, d\mu_{\e} \right)^{1/p},
\]
with the usual extension if $p = \infty$. We will also use $L^p$ norms of tensor fields with respect to different continuous metrics, such as $\g$, in which case we denote the norm by $\norm{\cdot}_{L^p(U, \g)}$, for example. Nevertheless, the Riemannian metrics that we use are all continuous and agree outside of a compact set, so the $L^p$ norms defined with different such metrics are equivalent. A tensor field $\T$ is said to lie in $L^p_{\loc}(M)$ if $\norm{\T}_{L^p(C)} < \infty$ for each compact set $C \subset M$. (The space $L^p_{\loc}(M)$ is independent of the metric in our class used to define the $L^p$ norms.) Similarly, for $k \ge 0$ an integer and $p \in [1, \infty]$, we introduce the Sobolev norm
\[
\norm{\T}_{W^{k, p}(U)}^p := \sum_{|\alpha| \le k} \int_U \norme{\nabla_{\e}^{\alpha} \T}^p \, d\mu_{\e},
\]
where $\nabla_{\e}$ denotes the Levi-Civita connection of the background metric $\e$. Again, a tensor field $\T$ lies in the local Sobolev space $W_{\loc}^{k, p}(M)$ if $\norm{\T}_{W^{k, p}(C)} < \infty$ for all compact $C \subset M$. Again, we will use Sobolev norms with respect to different smooth metrics, which we will indicate in the norm. The Sobolev norms will again be equivalent, since the smooth metrics that we consider will agree outside of a compact set, and the metrics and their derivatives will be bounded on the compact set.

\

We are now in a position to state the regularity conditions that we require on our metric:

\medskip

\noindent{\textbf{Regularity assumption}}: We assume that the metric $\g$ is smooth on the set $M \setminus K$. In addition, we assume that the metric $\g$ is continuous and lies in the local Sobolev space $W_{\loc}^{2, n/2}(M)$, i.e. $\g \in C^0(M) \cap W_{\loc}^{2, n/2}(M)$.

\begin{remark}
If we assume $\g \in W_{\loc}^{2, p}(M)$ for some $p > \frac{n}{2}$, then the Sobolev embedding theorem implies that $\g$ is automatically continuous. However, in the critical case where $\g \in W_{\loc}^{2, n/2}(M)$, $\g$ is generally not even locally bounded. Since we will require continuity of $\g$ (cf. Remark~\ref{rem:notLinfty}), it must therefore be added as an additional assumption.
\end{remark}

\smallskip

We first observe the following.

\begin{proposition}
\label{sinLn/2}
Let $\g \in C^0(M) \cap \Wloc{2, n/2}$. Then $s_{\g} \in \Lp{n/2}$.
\end{proposition}
\begin{proof}
The result is local, so we may perform the calculations in a local coordinate chart. Since $g_{..} \in W^{2, n/2}_{\loc}$, we have $\partial^2 g_{..} \in L^{n/2}_{\loc}$. The Sobolev embedding theorem implies that $\partial g_{..} \in W^{1, n/2}_{\loc} \subseteq L^n_{\loc}$. Schematically, the curvature tensor takes the form $R^{.}{}_{...} = g^{..} \partial^2 g_{..} + g^{..} \partial g_{..} \partial g_{..}$. Since $g^{..}$ is continuous, we deduce that both of the terms in this expression lie in $L^{n/2}_{\loc}$, therefore the curvature tensor of $\g$ lies in $L^{n/2}_{\loc}$. Since the scalar curvature follows from contracting the curvature tensor with the (continuous) inverse metric, it follows that the scalar curvature lies in $L^{n/2}_{\loc}$.
\end{proof}

\begin{remark}
The proposition remains valid if we make the weaker assumption that the metric lies in $\Lp{\infty} \cap \Wloc{2, n/2}$.
\end{remark}

\begin{remark}
\label{rem:SobolevBad}
We will only be concerned with scalar curvature bounds. Nevertheless, the proof of the proposition shows for metrics in $C^0(M) \cap \Wloc{2, n/2}$, the full curvature tensor lies in $\Lp{n/2}$. Ideally, we would like to consider metrics for which, for example, the scalar curvature is well-defined as a distribution, but the full curvature tensor may not be well-defined in this sense.
\end{remark}

\section{Non-negative scalar curvature and the positive mass theorem}
\label{sec:nnsc}

Let $\D(M)$ denote the collection of smooth, compactly supported test functions on $M$. We note from Lemma~\ref{sinLn/2} that $s_{\g} \in \Lp{n/2}$. It follows that the map $\D(M) \to \R$ defined by
\bel{sgdist}
\varphi \mapsto \la s_{\g}, \varphi \ra := \int_M s_{\g} \varphi \, d\mu_{\g}
\ee
is a well-defined distribution on $M$.

\begin{definition}
\label{def:snonneg}
We will say that the metric $\g$ has \emph{non-negative scalar curvature in the distributional sense\/} if $s_{\g} \ge 0$ in $\Dprime(M)$, i.e. for all $\varphi \in \D(M)$ with $\varphi \ge 0$, we have
\bel{nonnegscalar}
\la s_{\g}, \varphi \ra \ge 0.
\ee
\end{definition}

\smallskip
The main result of the first part of this paper is the following.

\begin{theorem}
\label{thm}
Let $(M, \g)$ be an asymptotically flat Riemannian manifold that satisfies our regularity conditions with non-negative scalar curvature in the distributional sense. Then the ADM mass of $(M, \g)$ is non-negative.
\end{theorem}

Our strategy of proof is as follows. In Lemma~\ref{lemma:smooth} below, we show that we can smooth out the metric $\g$ to yield a family of smooth metrics $\g_{\eps}$ that coincide with the metric $\g$ outside of a compact subset, and converge to $\g$ in $C^0(M) \cap W^{2, n/2}_{\loc}(M)$ as $\eps \to 0$. Since the $\g_{\eps}$ have the same asymptotic behaviour as $\g$, the ADM mass of the $\g_{\eps}$ is equal to that of $\g$. However, the metrics $\g_{\eps}$ will not, in general, have non-negative scalar curvature. Nevertheless, in Proposition~\ref{thm:W2p} below, we show that the $L^{n/2}$ norm of the negative part of $s_{\g_{\eps}}$ is bounded, and converges to zero as $\eps \to 0$. Following the approach of Miao~\cite{Miao}, we show in Section~\ref{sec:elliptic} that the metrics $\g_{\eps}$ may be conformally rescaled to give asymptotically flat metrics, $\ghat_{\eps}$, with non-negative scalar curvature. Since the metrics $\ghat_{\eps}$ are smooth and satisfy the conditions of the classical positive mass theorem, we deduce that $\adm(\ghat_{\eps}) \ge 0$. An analysis of the asymptotics of the $\ghat_{\eps}$ in Section~\ref{sec:mass} also shows that $\adm(\ghat_{\eps}) \to \adm(\g)$ as $\eps \to 0$. Hence $\adm(\g) \ge 0$, as required.

\

Before proceeding with the steps of our proof, some remarks are in order concerning Definition~\ref{def:snonneg}.

\smallskip

\begin{remark}
On $M \setminus K$, the metric $\g$ is smooth, so the scalar curvature $s_{\g}$ is a well-defined smooth quantity on this set. On $M \setminus K$, the condition~\eqref{nonnegscalar} is therefore equivalent to the condition that $s_{\g} \ge 0$ as a smooth function. However, inside the set $K$ the metric $\g$ is not assumed to be $C^2$, so~\eqref{nonnegscalar} implies non-negativity of the scalar curvature is only imposed in the weak sense.
\end{remark}

\begin{remark}
Since $s_{\g}$ lies in $L^{n/2}_{\loc}(M)$, \eqref{sgdist} implies that
\[
|\la s_{\g}, \varphi \ra| \le \norm{s_{\g}}_{L^{n/2}(K, \g)} \norm{\varphi}_{L^{\frac{n}{n-2}}(K, \g)},
\]
for all $\varphi \in \D(M)$ with support contained in a compact set $K$. It follows by density and continuity that the map~\eqref{sgdist} extends to the space of compactly supported functions in $L^{\frac{n}{n-2}}(M)$.
\end{remark}

\begin{remark}
In~\eqref{sgdist}, we have defined the action of $s_{\g}$ as a distribution using the measure $d\mu_{\g}$ for the rough metric $\g$. This seems appropriate, as it means that the distribution is independent of the background metric $\e$. Moreover, we will largely be concerned with the action of $s_{\g}$ on elements of $L^p$ spaces, which are independent of the continuous metric agreeing with $\g$ outside of a compact set used in the definition of the $L^p$ norm.
\end{remark}

\section{Smooth approximations of rough metrics}
\label{sec:approx}

Our method of proof involves smooth approximations of the metric $\g$. The following result provides us with smooth approximations with appropriate properties.

\begin{lemma}
\label{lemma:smooth}
For all $\eps > 0$, there exists a smooth Riemannian metric $\g_{\eps}$ and a compact set $K_{\eps} \subset M$ with the following properties:
\begin{enumerate}
\item\label{1} $\g_{\eps}$ converge to $\g$ locally uniformly and in $\Wloc{2, n/2}$ as $\eps \to 0$;
\item\label{2} $\g_{\eps}$ coincide with the metric $\g$ on the set $M \setminus K_{\eps}$;
\item\label{3} $K_{\eps}$ converge to $K$ as $\eps \to 0$.%
\end{enumerate}
\end{lemma}

\begin{proof}[Proof of Lemma~\ref{lemma:smooth}]
Generally, the existence of such a family of smooth approximating metrics follows from density of smooth metrics in $\Wloc{2, n/2} \cap C^0(M)$. More explicitly, we may proceed as follows. We cover $K$ by a finite collection of open coordinate charts $\psi_i \colon O_i \to B(0, 1) \subset \R^n$, $i = 1, \dots, m$ with the property that $K \subset \cup_{i=1}^m O_i$ and $M = N \cup \left( \cup_{i=1}^m O_i \right)$, where $N := M \setminus K$. Let $\chi_i$, $i = 1, \dots, m$ and $\chi_N$ be a smooth squared partition of unity (i.e. $\chi_N^2+\sum_{i=1}^m{\chi_i^2}=1$) subordinate to the cover of $M$ defined by $O_i$ and $N$ with the property that the functions $\chi_i \circ \psi_i^{-1}$ have compact support contained in the set $B(0, 1)$ and that $\chi_N$ has support bounded away from $\partial N$. We decompose the metric using the partition of unity, letting $\g_i := \chi_i \g$, $i = 1, \dots, m$ and $\g_N := \chi_N \g$ be $(0, 2)$ tensor fields on $M$ with support contained in $O_i$ and $N$, respectively. For $i = 1, \dots, m$, we define the $(0, 2)$ tensor field $\mathbf{G}_i := \left( \psi_i^{-1} \right)^* \g_i$ on $B(0, 1) \equiv \psi_i (O_i)$, which has compact support contained away from the boundary of $B(0, 1)$. In terms of the coordinates $x_i^{\alpha}$ on $\psi_i(O_i) \subset \R^n$, the components $G_{i, \alpha\beta}$ of $\mathbf{G}_i$ are continuous and lie in $W_{\loc}^{2, n/2}(B(0, 1))$. Let $\rho \colon \R^n \to \R$ be a smooth, positive mollifier with $\supp \rho \subset B(0, 1)$ and $\int_{B(0, 1)} \rho = 1$. Let $\rho_{\eps}(\x) := \eps^{-n} \rho \left( \frac{\x}{\eps} \right)$. We construct smoothed versions of the tensor fields $\mathbf{G}_i$ by taking the scaled convolution of the components $G_{i,\eps; \alpha\beta}(x) := \left( \rho_{\eps} \star G_{i, \alpha\beta} \right)(x)$ for $x \in B(0, 1)$. Since $\mathbf{G}_i$ has support bounded away from $\partial B(0, 1)$, it follows that there exists $\eps_0 > 0$ such that, for all $\eps < \eps_0$, the $G_{i, \eps; \alpha\beta}$ will have support bounded away from $\partial B(0, 1)$ for $i = 1, \dots, m$. From the components $G_{i, \eps; \alpha\beta}$, we now reconstruct the smooth $(0, 2)$ tensor fields $\mathbf{G}_{i, \eps}$ on $B(0, 1)$. Let $\g_{i, \eps} := \left( \psi_i \right)^* \mathbf{G}_{i, \eps}$, for $i = 1, \dots, m$. We now define the $(0, 2)$ tensor field on $M$
\[
\g_{\eps} := \chi_N \g_N + \sum_{i = 1}^m \chi_i \g_{i, \eps}.
\]
(Note that we have not changed the (smooth) metric $\g_N$.) By the convergence properties of smoothing with mollifiers, the $(0, 2)$ tensor fields $\g_{i, \eps}$ will converge to $\g_i$ both locally uniformly and in $\Wloc{2, n/2}$ as $\eps \to 0$. Since $\g_i = \chi_i \g$, we therefore have that, both locally uniformly and in $\Wloc{2, n/2}$,
\[
\g_{\eps} \to \chi_N \g_N + \sum_{i = 1}^m \chi_i \g_i = \chi_N^2 \g + \sum_{i=1}^m \chi_i^2 \g = \g,
\]
since $(\chi_1, \dots, \chi_m, \chi_N)$ is a squared partition of unity. Therefore, Condition~(\ref{1}) is satisfied. Since we have not modified the metric $\g_N$, and the $\g_{i, \eps}$ only differ from $\g_i$ on a set of size $\eps$, it follows that the $\g_{\eps}$ will coincide with $\g$ on an $\eps$-neighbourhood of the set $K$. We define $K_{\eps}$ to be the closure of this set, which is automatically compact. With this definition of $K_{\eps}$, Conditions~(\ref{2}) and~(\ref{3}) are satisfied.
\end{proof}

\begin{proposition}
\[
\adm(\g_{\eps}) = \adm(\g).
\]
\end{proposition}
\begin{proof}
By construction, $\g_{\eps} = \g$ on $M \setminus K_{\eps}$. Since the ADM mass is determined by the asymptotic properties of a metric, the result follows.
\end{proof}

\begin{lemma}
\label{lemma:rho}
There exists $\rho(\eps) \ge 1$ with the property that
\bel{equivalence}
\frac{1}{\rho(\eps)} \g_{\eps} \le \g \le \rho(\eps) \g_{\eps}
\ee
as bilinear forms on $M$, with $\rho(\eps) \to 1$ as $\eps \to 0$.
\end{lemma}
\begin{proof}
By construction, the $\g_{\eps}$ converge uniformly to $\g$ on compact subsets. Taking the compact set to be $K_{\eps}$, it follows that there exists $\rho(\eps)$ such that~\eqref{equivalence} holds on the set $K_{\eps}$. Since $\g_{\eps}$ coincides with $\g$ on $M \setminus K_{\eps}$, it follows that~\eqref{equivalence} holds globally on $M$ for each $\eps > 0$. Uniform convergence of the metrics on $K$ and the fact that $K_{\eps} \to K$ as $\eps \to 0$ implies that $\rho(\eps) \to 1$ as $\eps \to 0$.
\end{proof}

\begin{proposition}
\label{W2pconvergence}
Let $\g \in C^0(M) \cap \Wloc{2, n/2}$ and $\geps$ as in Lemma~\ref{lemma:smooth}. Then $s_{\geps} \to s_{\g}$ in $\Lp{n/2}$ as $\eps \to 0$.
\end{proposition}
\begin{proof}
Again, this is a local calculation. Since $g_{..} \in W^{2, n/2}_{\loc}$, we have $\partial^2 g_{\eps ..} \to \partial^2 g_{..} \in L^{n/2}_{\loc}$. By the Sobolev embedding theorem, $\partial \g_{..} \in L^n_{\loc}$, so $\partial g_{\eps, ..} \to \partial g_{..}$ in $L^n_{\loc}$. Finally, since $g_{..}$ are assumed continuous, the $g_{\eps, ..}$ converge locally uniformly to $g_{..}$ as $\eps \to 0$. Given any compact set $C$, we then have
\begin{align}
\norm{R[\geps]^{.}{}_{...} - R[\g]^{.}{}_{...}}_{L^{n/2}(C)}
&\sim \norm{g_{\eps}^{..} \partial^2 g_{\eps ..} + g_{\eps}^{..} \partial g_{\eps ..} \partial g_{\eps ..} - g^{..} \partial^2 g_{..} - g^{..} \partial g_{..} \partial g_{..}}_{L^{n/2}(C)}
\nonumber
\\
&\le \norm{g_{\eps}^{..} \partial^2 g_{\eps, ..} - g^{..} \partial^2 g_{..}}_{L^{n/2}(C)}
+ \norm{g_{\eps}^{..} \partial g_{\eps, ..} \partial g_{\eps ..} - g^{..} \partial g_{..} \partial g_{..}}_{L^{n/2}(C)}
\nonumber
\\
&\le \norm{ g_{\eps}^{..} \left( \partial^2 g_{\eps ..} - \partial^2 g_{..} \right)}_{L^{n/2}(C)}
+ \norm{\left( g_{\eps}^{..} - g^{..} \right) \partial^2 g_{..}}_{L^{n/2}(C)}
\nonumber
\\
&\hskip 1cm + \norm{g_{\eps}^{..} \partial g_{\eps, ..} \left( \partial g_{\eps ..} - \partial g_{..} \right)}_{L^{n/2}(C)}
+ \norm{g_{\eps}^{..} \left( \partial g_{\eps ..} - \partial g_{..} \right) \partial g_{..}}_{L^{n/2}(C)}
\nonumber
\\
&\hskip 1cm + \norm{\left( g_{\eps}^{..} - g^{..} \right) \partial g_{..} \partial g_{..}}_{L^{n/2}(C)}
\nonumber
\\
&=: I + II + III + IV + V.
\label{12345}
\end{align}
We then have
\begin{align*}
I &\le \norm{g_{\eps}^{..}}_{L^{\infty}(C)} \norm{\partial^2 g_{\eps ..} - \partial^2 g_{..}}_{L^{n/2}(C)},
\\
II &\le \norm{g_{\eps}^{..} - g^{..}}_{L^{\infty}(C)} \norm{\partial^2 g_{..}}_{L^{n/2}(C)},
\\
III &\le \norm{g_{\eps}^{..}}_{L^{\infty}(C)} \norm{\partial g_{\eps ..}}_{L^n(C)} \norm{\partial g_{\eps ..} - \partial g_{..}}_{L^{n}(C)},
\\
IV &\le \norm{g_{\eps}^{..}}_{L^{\infty}(C)} \norm{\partial g_{\eps ..} - \partial g_{..}}_{L^n} \norm{\partial g_{..}}_{L^n(C)},
\\
V &\le \norm{g_{\eps}^{..} - g^{..}}_{L^{\infty}(C)} \norm{\partial g_{..}}_{L^n(C)}^2.
\end{align*}
From these estimates, it follows that each term on the right-hand-side of~\eqref{12345} converges to $0$ as $\eps \to 0$. Therefore, the curvature of the metrics $\geps$ converges to that of $\g$ in $\Lp{n/2}$ as $\eps \to 0$. Since $\geps$ converge locally uniformly to $\g$ as $\eps \to 0$, it follows that $s_{\geps} \to s_{\g}$ in $\Lp{n/2}$ as $\eps \to 0$.
\end{proof}

\begin{remark}
Further to Remark~\ref{rem:SobolevBad}, we note that the full curvature tensor of the metrics $\geps$ converges to that of $\g$ in $\Lp{n/2}$. It therefore appears that Sobolev spaces may be a rather blunt tool with which to study scalar curvature bounds. It would be of interest to know whether more appropriate spaces of metrics, perhaps defined by isoperimetric conditions and/or volume doubling conditions, may be more analytically suited to the study of low-regularity metrics with non-negative scalar curvature (see, for instance, \cite{ACM}).
\end{remark}

Specifically regarding the negative part of the scalar curvature of the metrics $\g_{\eps}$, we have the following result.

\begin{proposition}
\label{thm:W2p}
Let $\g \in C^0(M) \cap \Wloc{2, n/2}$, and have non-negative scalar curvature in the distributional sense. Then the negative part of the scalar curvature of the metric $\geps$ satisfies
\bel{smin0}
\left\Vert \sepsmin \right\Vert_{L^{n/2}(M, \g)}
\le \left\Vert s_{\g_{\eps}} - s_{\g} \right\Vert_{L^{n/2}(M, \g)}
\to 0 \quad \mbox{as $\eps \to 0$}.
\ee
In particular, $\sepsmin \to 0$ in $L^{n/2}(M, \g)$ as $\eps \to 0$.
\end{proposition}
\begin{proof}
First note that, since $\g_{\eps} = \g$ outside of the compact set $K_{\eps}$, the quantities $\sepsmin$ and $s_{\g_{\eps}} - s_{\g}$ have support contained in $K_{\eps}$. As such, the integrals appearing in~\eqref{smin0} are well-defined.

By assumption, given $\varphi \in \D(M)$ with $\varphi \ge 0$, we have
\begin{align*}
\int_M s_{\geps} \varphi \, d\mu_{\g}
&= \la s_{\g}, \varphi \ra + \int_M \left( s_{\geps} - s_{\g} \right) \varphi \, d\mu_{\g}
\\
&\ge \int_M \left( s_{\geps} - s_{\g} \right) \varphi \, d\mu_{\g}
\\
&\ge - \left\Vert s_{\geps} - s_{\g} \right\Vert_{L^{n/2}(K_{\eps}, \g)} \left\Vert \varphi \right\Vert_{L^{\frac{n}{n-2}}(K_{\eps}, \g)}.
\end{align*}
Let $A_{\geps}$ denote $\left\Vert s_{\geps} - s_{\g} \right\Vert_{L^{n/2}(K_{\eps}, \g)}$. We therefore have that, for all $\varphi \in \D(M)$ with $\varphi \ge 0$,
\[
\int_M \sepsmin \varphi \, d\mu_{\g}
\le \int_M (s_{\geps})_+ \varphi \, d\mu_{\g} + A_{\geps} \left\Vert \varphi \right\Vert_{L^{\frac{n}{n-2}}(M, \g)}.
\]
Without loss of generality, we assume that the set $\supp \sepsmin \subseteq K_{\eps}$ has non-zero measure (otherwise, there is nothing to prove). Given any non-negative test function $\varphi$ with $\supp \varphi \subseteq \sepsmin$, we therefore have
\bel{ssss}
\int_M \sepsmin \varphi \, d\mu
\le A_{\geps} \left\Vert \varphi \right\Vert_{L^{\frac{n}{n-2}}(M, \g)}.
\ee
The proof now follows the proof of the extremal version of {\H}'s inequality. Since $\sepsmin \in L^{n/2}_{\loc}(M, \g)$, this inequality extends by density and continuity to test objects $\varphi \in L^{\frac{n}{n-2}}(M, \g)$ such that $\varphi \ge 0$ with $\supp \varphi \subseteq \supp \sepsmin$. Taking the admissible test function $\varphi := \left( \sepsmin \right)^{\frac{n}{2} - 1}$ in~\eqref{ssss} yields $\left\Vert \sepsmin \right\Vert_{L^{n/2}(M, \g)} \le A_{\geps}$, as required.
\end{proof}

\begin{remark}
\label{rem:notLinfty}
If the metric $\g$ were only assumed to lie in $L_{\loc}^{\infty}(M) \cap W^{2, n/2}_{\loc}(M)$, then, generally, $\g_{\eps}$ would not converge to $\g$ in $L_{\loc}^{\infty}$, so the curvature of the $\g_{\eps}$ would not converge to that of $\g$ in $L_{\loc}^{n/2}(M)$. This is the reason that we must assume that the metric $\g$ is continuous.
\end{remark}

\begin{remark}
For the class of metrics discussed in~\cite{Miao, McS, Lee}, one obtains a pointwise bound on the negative part of $s_{\geps}$. For our metrics, it is an $L^{n/2}$ bound on $\sepsmin$ that appears naturally. In the Ricci flow approach adopted in~\cite{McS}, a pointwise lower bound on the scalar curvature of the smoothed metric is required in order to ensure that the solution of the flow with the metric $\g$ as initial data has non-negative scalar curvature. It would be of interest to know whether the Ricci flow approach may be adapted to metrics in our class, especially since this approach appears to be well-suited to proving rigidity theorems, which we are unable to establish with our techniques.
\end{remark}

\subsection{Sobolev constants}
Given a real number $k > 1$, we let $k^* := \frac{2k}{k-2}$. In particular, $n^* := \frac{2n}{n-2}$. We recall the following version of the Sobolev inequality from~\cite[Lemma~3.1]{SY}.

\begin{lemma}
\label{lemma:SY}
For each metric $\g_{\eps}$, there exists a constant $C > 0$ such that for any function $\varphi$ with compact support on $M$ we have
\bel{SYSob}
\norm{\varphi}_{L^{n^*}(M, \g_{\eps})}^2 \le C \, \norm{\nabla_{\g_{\eps}} \varphi}_{L^2(M, \g_{\eps})}^2.
\ee
The smallest such constant will be denoted $c_1[\g_{\eps}]$ and referred to as the \emph{Sobolev constant of $\geps$}.
\end{lemma}

An inspection of the proof of this Lemma in~\cite{SY} shows that this result remains valid for continuous metrics. We therefore have that, for any compactly supported function on $M$,
\[
\norm{\varphi}_{L^{n^*}(M, \g)}^2 \le c_1[\g] \, \norm{\nabla_{\g} \varphi}_{L^2(M, \g)}^2.
\]
where $c_1[\g] > 0$ denotes the Sobolev constant of the metric $\g$.

\begin{proposition}
The Sobolev constants of the metrics $\g_{\eps}$ are related to that of the metric $\g$ by the inequality
\bel{Sobequiv}
\frac{1}{\rho(\eps)^n} \, c_1[\g_{\eps}] \le c_1[\g] \le \rho(\eps)^n \, c_1[\g_{\eps}],
\ee
where $\rho$ is the function introduced in Lemma~\ref{lemma:rho}.
\end{proposition}
\begin{proof}
Let $\varphi$ be a function with compact support on $M$. We then have
\begin{align*}
\norm{\varphi}_{L^{n^*}(M, \g)}^2 &= \left( \int_M |\varphi|^{n^*} \, d\mu_{\g} \right)^{2/n^*}
\\
&\le \left( \rho(\eps)^{n/2} \int_M |\varphi|^{n^*} \, d\mu_{\g_{\eps}} \right)^{2/n^*}
\\
&\le \rho(\eps)^{\frac{n-2}{2}} c_1[\g_{\eps}] \left( \int_M |\nabla\varphi|_{\g_{\eps}}^2 \, d\mu_{\g_{\eps}} \right)
\\
&\le \rho(\eps)^{\frac{n-2}{2}} c_1[\g_{\eps}] \left( \rho(\eps)^{n/2+1} \int_M |\nabla\varphi|_{\g}^2 \, d\mu_{\g} \right)
\\
&= \rho(\eps)^n c_1[\g_{\eps}] \, \norm{\nabla\varphi}_{L^2(M, \g)}^2.
\end{align*}
Therefore $c_1[\g] \le \rho(\eps)^n c_1[\g_{\eps}]$. Reversing the same argument gives the other part of~\eqref{Sobequiv}.
\end{proof}

\section{Conformal transformations and elliptic estimates}
\label{sec:elliptic}

The approximating metrics $\g_{\eps}$ constructed in the previous section do not necessarily have non-negative scalar curvature. However, as shown in Proposition~\ref{thm:W2p}, we have control over the $L^{n/2}$ norm of the negative part of $s_{\g_{\eps}}$. We will follow an approach of Miao~\cite{Miao} in performing a conformal transformation of the metrics $\g_{\eps}$ to obtain asymptotically flat metrics $\ghat_{\eps}$ with non-negative scalar curvature. Since the results of this section are of some independent interest, we develop our estimates for an arbitrary metric, applying them to the case of the $\g_{\eps}$ in Section~\ref{epsilonsbackin}. Therefore, let $\h$ be an arbitrary smooth asymptotically flat, Riemannian metric on the manifold $M$.

\smallskip

We will require the following result of Schoen and Yau.

\begin{lemma}{\cite[Lemma~3.2]{SY}}
\label{lemma:SY2}
There exists a constant $\eps_0 = \eps_0(\h) > 0$ with the property that if
\[
\norm{f_-}_{L^{n/2}(M, \h)} \le \eps_0(\h),
\]
then the partial differential equation
\[
\Delta_{\h} u - f u = F
\]
has a unique solution on $M$ that satisfies%
\footnote{Here, and below, $\x$ refer to the asymptotic coordinate system in the definition of the end $N$.}
$u(\x) = \frac{A}{|\x|^{n-2}} + \omega(|\x|)$ as $|\x| \to \infty$, where $A$ is constant and $\omega(|\x|) = O(|\x|^{1-n})$ as $|\x| \to \infty$.
\end{lemma}

We recall, and make more explicit, some parts of the proof of Lemma~\ref{lemma:SY2} from~\cite{SY}, since we will later require rather precise information concerning the constant $\eps_0(\h)$. Let $\Omega$ be an open subset of $M$ with compact closure, such that $\partial\Omega$ is smooth and contained in the ends of $M$. The proof of Lemma~\ref{lemma:SY2} involves solving the Dirichlet problem
\bel{SYproof}
\Delta_{\h} v - f v = F, \qquad \left. v \right|\partial\Omega = 0
\ee
on the set $\Omega$. Treating the homogeneous problem with $F = 0$ and using the Sobolev inequality~\eqref{SYSob} for the metric $\h$ yields the inequality
\begin{align*}
\int_{\Omega} |\nabla v|_{\h}^2 \, d\mu_{\h}
&\le c_1[\h] \, \norm{f_-}_{L^{n/2}(\Omega, \h)} \, \int_{\Omega} |\nabla v|_{\h}^2\, d\mu_{\h}
\\
&\le c_1[\h] \, \norm{f_-}_{L^{n/2}(M, \h)} \, \int_{\Omega} |\nabla v|_{\h}^2\, d\mu_{\h}.
\end{align*}
It follows that if the ($\Omega$ independent) condition
\bel{fcondition}
c_1[\h] \, \norm{f_-}_{L^{n/2}(M, \h)} < 1,
\ee
is satisfied, then the homogeneous problem~\eqref{SYproof} has a unique smooth solution $v \equiv 0$ on the region $\Omega$. Fredholm theory then yields the (unique) existence of a solution of the inhomogeneous problem~\eqref{SYproof}.

Let $u$ denote a solution of the equation
\[
\Delta_{\h} u + \frac{1}{a_n} \shmin u = 0,
\]
where
\[
a_n := 4 \, \frac{n-1}{n-2}.
\]
The conformal transformation formula for the scalar curvature (see, for instance, \cite{LeeParker}) implies that the conformally rescaled metric
\[
\hhat := u^{4/(n-2)} \h
\]
has scalar curvature satisfying
\begin{align*}
\frac{1}{a_n} s_{\hhat} u^{\frac{n+2}{n-2}}
= - \Delta_{\h} u + \frac{1}{a_n} s_{\h} u
= - \left( \Delta_{\h} u + \frac{1}{a_n} \shmin u \right) + \frac{1}{a_n} (s_{\h})_+ u.
\end{align*}
As such, if we impose that the conformal factor $u$ satisfies the equation $\Delta_{\h} u + \frac{1}{a_n} \shmin u = 0$ in $\Omega$ with $u = 1$ on $\partial\Omega$, then the metric $\hhat$ will have non-negative scalar curvature on $\Omega$. Letting $u := 1 + v$, we require that $v$ satisfies the equation
\bel{veqn}
\Delta_{\h} v + \frac{1}{a_n} \shmin \, v + \frac{1}{a_n} \shmin = 0 \mbox{ in $\Omega$}, \qquad \left. v \right| \partial\Omega = 0.
\ee

\begin{proposition}
\label{prop:bounds}
Let $v$ satisfy~\eqref{veqn}. We define the quantity
\[
\alpha := \frac{1}{a_n} c_1[\h] \, \left\Vert \shmin \right\Vert_{L^{n/2}(M, \h)}.
\]
We assume that the condition $\alpha < 1$ is satisfied, and that the function $\shmin$ has compact support. Then the solution $v$ has the following properties
\begin{subequations}
\begin{align}
\norm{v}_{L^{n^*}(\Omega, \h)}
&\le \frac{\alpha}{1 - \alpha} \, \mu_{\h} \left( \Omega \cap \supp \shmin \right)^{1/n^*},
\label{vbound1b}
\\
\norm{\nabla v}_{L^2(\Omega, \h)}^2
&\le \frac{1}{a_n} \, \frac{\alpha}{\left( 1 - \alpha \right)^2} \,
\left\Vert \shmin \right\Vert_{L^{n/2}(M, \h)} \, \mu_{\h} \left( \Omega \cap \supp \shmin \right)^{2/n^*},
\label{dwbound1b}
\end{align}\end{subequations}
where, for any measurable set $E$, $\mu_{\h}(E) := \int_E d\mu_{\h}$ denotes the volume of $E$ with respect to the Riemannian volume element.
\end{proposition}
\begin{proof}
Our proof is a modification of that of Proposition~4.1 in~\cite{Miao}. Let $f$ denote the function $\frac{1}{a_n} \shmin$. Multiplying the equation $\Delta_{\h} v + f v + f = 0$ by $v$, integrating over $\Omega$, and using {\H}'s inequality, we deduce that
\begin{align}
\norm{\nabla v}_{L^2(\Omega, \h)}^2 &= \int_{\Omega} {\left( f v^2 + f v \right) \, d\mu_{\h}}
\nonumber
\\
&\le \norm{f}_{L^{n/2}(\Omega, \h)} \norm{v}_{L^{n^*}(\Omega, \h)}^2 + \norm{f}_{L^{n/2}(\Omega, \h)} \norm{1}_{L^{n^*}(\Omega, \h)}\ \norm{v}_{L^{n^*}(\Omega, \h)}\
\nonumber
\\
&\le \norm{f}_{L^{n/2}(M, \h)} \norm{v}_{L^{n^*}(\Omega, \h)}^2 + \norm{f}_{L^{n/2}(M, \h)} \mu_{\h}(\Omega \cap \supp f)^{1/n^*} \norm{v}_{L^{n^*}(\Omega, \h)}.
\label{M1}
\end{align}
It follows from the Sobolev inequality that we have
\begin{align*}
\norm{v}_{L^{n^*}(\Omega, \h)}^2 &\le c_1[\h] \norm{f}_{L^{n/2}(M, \h)} \norm{v}_{L^{n^*}(\Omega, \h)}^2
+ c_1[\h] \norm{f}_{L^{n/2}(M, \h)} \mu_{\h}(\Omega \cap \supp f)^{1/n^*} \norm{v}_{L^{n^*}(\Omega, \h)}
\\
&\equiv \alpha \norm{v}_{L^{n^*}(\Omega, \h)}^2 + \alpha \mu_{\h}(\Omega \cap \supp f)^{1/n^*} \norm{v}_{L^{n^*}(\Omega, \h)}
\end{align*}
Therefore, since $\alpha < 1$, we deduce that
\[
\norm{v}_{L^{n^*}(\Omega, \h)} \le \frac{\alpha}{1 - \alpha} \mu_{\h}(\Omega \cap \supp f)^{1/n^*},
\]
as required. The inequality~\eqref{dwbound1b} now follows from~\eqref{M1} and~\eqref{vbound1b}.
\end{proof}

\begin{remark}
The proof of Proposition~\ref{prop:bounds} may be tightened to give the slightly stronger result that if
\bel{badbd}
c_1[\h] \norm{\shmin}_{L^{n/2}(\Omega, \h)} < a_n,
\ee
then we have the inequalities
\begin{subequations}
\begin{align}
\norm{v}_{L^{n^*}(\Omega, \h)}
&\le \frac{c_1[\h] \, \norm{\shmin}_{L^{n/2}(\Omega, \h)}}{a_n - c_1[\h] \, \norm{\shmin}_{L^{n/2}(\Omega, \h)}} \mu_{\h}(\Omega \cap \supp f)^{1/n^*},
\label{vbound1}
\\
\norm{\nabla v}_{L^2(\Omega, \h)}^2
&\le
\frac{c_1[\h] \norm{\shmin}_{L^{n/2}(\Omega, \h)}^2}{\left( a_n - c_1[\h] \norm{ \shmin}_{L^{n/2}(\Omega, \h)} \right)^2} \mu_{\h}(\Omega \cap \supp f)^{2/n^*},
\label{dwbound1}
\end{align}\end{subequations}
These estimates are less useful since the right-hand-sides, and the condition~\eqref{badbd} on $\shmin$, depend on the set $\Omega$.
\end{remark}

Taking a sequence of sets $\Omega$ with $\overline{\Omega}$ a compact exhaustion of $M$, then the condition~\eqref{fcondition} (being independent of $\Omega$) implies existence of solutions of the corresponding Dirichlet problems. We then extract a subsequence of solutions that converges to the solution on $M$ referred to in Lemma~\ref{lemma:SY2}. Combining the Lemma~\ref{lemma:SY2} with the results of Proposition~\ref{prop:bounds}, we have the following result.
\begin{theorem}
\label{thm:SUSY}
If $c_1[\h] \norm{\shmin}_{L^{n/2}(M, \h)} < a_n$, then there exists a unique smooth solution to the partial differential equation
\[
\Delta_{\h} v + \frac{1}{a_n} \shmin v + \frac{1}{a_n} \shmin = 0
\]
with the property that $v(\x) = \frac{A}{|\x|^{n-2}} + O(|\x|^{1-n})$ as $|\x| \to \infty$, where $A$ is constant. This solution has the properties that
\begin{subequations}
\begin{align}
\norm{v}_{L^{n^*}(M, \h)}
&\le \frac{c_1[\h] \, \norm{\shmin}_{L^{n/2}(M, \h)}}{a_n - c_1[\h] \, \norm{\shmin}_{L^{n/2}(M, \h)}} \, \mu_{\h}(\supp \shmin )^{1/n^*},
\label{vbound1c}
\\
\norm{\nabla v}_{L^2(M, \h)}^2
&\le \frac{ c_1[\h] \norm{\shmin}_{L^{n/2}(M, \h)}^2}{\left( a_n - c_1[\h] \norm{\shmin}_{L^{n/2}(M, \h)} \right)^2} \, \mu_{\h}(\supp \shmin)^{2/n^*},
\label{dwbound1c}
\end{align}\end{subequations}
\end{theorem}

\subsection{Estimates for the metrics $\g_{\eps}$}
\label{epsilonsbackin}

We now apply the above results for the metrics $\g_{\eps}$. We replace $\h$ in Theorem~\ref{thm:SUSY} with $\g_{\eps}$, denoting the solution referred to there by $v_{\eps}$. We first note the following.

\begin{proposition}
\label{prop:previous}
\[
c_1[\g_{\eps}] \norm{\sepsmin}_{L^{n/2}(M, \g_{\eps})} \to 0 \quad \mbox{ as } \quad \eps \to 0.
\]
\end{proposition}
\begin{proof}
From~\eqref{Sobequiv} and~\eqref{equivalence}, we have
\[
c_1[\g_{\eps}] \norm{\sepsmin}_{L^{n/2}(M, \g_{\eps})}
\le \rho(\eps)^n c_1[\g] \cdot \rho(\eps) \norm{\sepsmin}_{L^{n/2}(M, \g_{\eps})}.
\]
Since $\rho(\eps) \to 1$ as $\eps \to 0$, the result now follows from~\eqref{smin0}.
\end{proof}

Theorem~\ref{thm:SUSY} and Proposition~\ref{prop:previous} immediately yield the following.
\begin{corollary}
There exists $\eps_0 > 0$ such that for all $\eps < \eps_0$, there exists a unique solution of
\bel{vepeqn}
\Delta_{\g_{\eps}} v_{\eps} + \frac{1}{a_n} \sepsmin v_{\eps} + \frac{1}{a_n} \sepsmin = 0
\ee
with the property that
\bel{vepasymp}
v_{\eps}(\x) = \frac{A_{\eps}}{|\x|^{n-2}} + O(|\x|^{1-n}) \mbox{ as $|\x| \to \infty$},
\ee
where $A_{\eps}$ is a constant. This solution has the properties that
\begin{subequations}
\begin{align}
\norm{v_{\eps}}_{L^{n^*}(M, \g_{\eps})}
&\le \frac{c_1[\g_{\eps}] \, \norm{\sepsmin}_{L^{n/2}(M, \g_{\eps})}}{a_n - c_1[\g_{\eps}] \, \norm{\sepsmin}_{L^{n/2}(M, \g_{\eps})}} \, \mu_{\g_{\eps}}(\supp \sepsmin )^{1/n^*},
\label{vbound1d}
\\
\norm{\nabla v_{\eps}}_{L^2(M, \g_{\eps})}^2
&\le \frac{c_1[\g_{\eps}] \norm{\sepsmin}_{L^{n/2}(M, \g_{\eps})}^2}{\left( a_n - c_1[\g_{\eps}] \norm{\sepsmin}_{L^{n/2}(M, \g_{\eps})} \right)^2} \, \mu_{\g_{\eps}}(\supp \sepsmin)^{2/n^*},
\label{dwbound1d}
\end{align}\label{boundtogether}\end{subequations}
\end{corollary}

Hence:
\begin{theorem}
\label{thm:gotozero}
If $\g$ has non-negative scalar curvature in the distributional sense, then the solutions of~\eqref{vepeqn} have the property that
\[
v_{\eps} \to 0 \mbox{ in } L^{n^*}(M, \g_{\eps}) \mbox{ as $\eps \to 0$},
\]
and
\[
\nabla v_{\eps} \to 0 \mbox{ in } L^2(M, \g_{\eps}) \mbox{ as $\eps \to 0$}.
\]
\end{theorem}
\begin{proof}
This follows directly from~\eqref{boundtogether} and the fact that $\norm{\sepsmin}_{L^{n/2}(M, \g_{\eps})} \to 0$ as $\eps \to 0$.
\end{proof}

\section{Mass calculations}
\label{sec:mass}

\begin{proposition}
\label{conformal}
The conformally rescaled metrics ${\ghat}_{\eps} := u_{\eps}^{4/(n-2)} \g_{\eps}$ are asymptotically flat and have non-negative scalar curvature. Furthermore, $\adm(\ghat_{\eps}) \to \adm(\g_{\eps})$ as $\eps \to 0$.
\end{proposition}

\begin{proof}
The scalar curvature of the metric ${\ghat}_{\eps}$ is non-negative by construction. The asymptotics of the function $u_{\eps} \equiv 1 + v_{\eps}$ given in Theorem~\ref{thm:SUSY} imply that ${\ghat}_{\eps}$ is asymptotically flat, and that the mass of $(M, \ghat_{\eps})$ is related to that of $(M, \g_{\eps})$ by the relation
\[
\adm(\ghat_{\eps}) = \adm(\g_{\eps}) + 2 A_{\eps},
\]
where $A_{\eps}$ is the constant that appears in the asymptotic expansion~\eqref{vepasymp}. The result will therefore follow if we can show that $A_{\eps} \to 0$ as $\eps \to 0$. From the fact that $u_{\eps}$ satisfies $\Delta_{\g_{\eps}} u_{\eps} + \frac{1}{a_n} \sepsmin u_{\eps} = 0$ and the asymptotics of $u_{\eps}$, it follows that
\[
A_{\eps} = - \frac{1}{(n-2) \omega_{n-1}} \int_M \left[ |\nabla_{\g_{\eps}} u_{\eps}|_{\g_{\eps}}^2 - \frac{1}{a_n} \sepsmin u_{\eps}^2 \right] \, d\mu_{\g_{\eps}}.
\]
Theorem~\ref{thm:gotozero} implies that the first term in the integral on the right-hand-side converges to $0$ as $\eps \to 0$. For the second term, we have
\[
\left| \int_M \sepsmin u_{\eps}^2 \, d\mu_{\g_{\eps}} \right|
= \left| \int_{K_{\eps}} \sepsmin u_{\eps}^2 \, d\mu_{\g_{\eps}} \right|
\le \norm{\sepsmin}_{L^{n/2}(K_{\eps}, \g_{\eps})} \norm{u_{\eps}}_{L^{n^*}(K_{\eps}, \g_{\eps})}^2.
\]
Moreover,
\begin{align*}
\norm{u_{\eps}}_{L^{n^*}(K_{\eps}, \g_{\eps})}
&\le \norm{1}_{L^{n^*}(K_{\eps}, \g_{\eps})} + \norm{v_{\eps}}_{L^{n^*}(K_{\eps}, \g_{\eps})}
\\
&= \mu_{\g_{\eps}}(K_{\eps})^{1/n^*} + \norm{v_{\eps}}_{L^{n^*}(K_{\eps}, \g_{\eps})} \to \mu_{\g}(K)^{1/n^*} \quad \mbox{as $\eps \to 0$},
\end{align*}
by Theorem~\ref{thm:gotozero}. From compactness of $K$, we have $\mu_{\g}(K) < \infty$. Since $\norm{\sepsmin}_{L^{n/2}(K, \g_{\eps})} \to 0$, it follows that $\int_M \sepsmin u_{\eps}^2 \, d\mu_{\g_{\eps}} \to 0$ as $\eps \to 0$. Therefore $\adm(\ghat_{\eps}) \to \adm(\g_{\eps})$ as $\eps \to 0$.
\end{proof}

\begin{proof}[Completion of the proof of Theorem~\ref{thm}]
Since $\adm(\g_{\eps}) = \adm(\g)$, we deduce from Proposition~\ref{conformal} that $\adm(\ghat_{\eps}) \to \adm(\g)$ as $\eps \to 0$. The classical positive mass theorem implies that $\adm(\ghat_{\eps}) \ge 0$ for $\eps > 0$. Therefore, $\adm(\g) \ge 0$.
\end{proof}

\subsection{Lower bound on the mass}
Our techniques also allow us to treat metrics for which the scalar curvature is \emph{not\/} non-negative, but the negative part of the scalar curvature has sufficiently small $L^{n/2}$ norm. In this case, we may derive a (negative) lower bound on the mass of $(M, \g)$. To our knowledge, this is the first known result of this type.

\begin{theorem}
Let $\g$ be an asymptotically flat metric with $\sgmin$ of compact support and satisfying the condition that
\[
c_1[\g] \, \norm{\sgmin}_{L^{n/2}(M, \g)} < a_n,
\]
where
\[
a_n := 4 \, \frac{n-1}{n-2}.
\]
Then the mass of the metric $\g$ satisfies
\bel{lowerbdonmass}
\adm(\g)
\ge - \frac{1}{2 (n-1) \omega_{n-1}} \frac{\norm{\sgmin}_{L^{n/2}(M, \g)}}{\left( 1 - \frac{1}{a_n} c_1[\g] \, \norm{\sgmin}_{L^{n/2}(M, \g)} \right)^2} \, \mu_{\g}(\supp (\sgmin)^{2/n^*}.
\ee
\end{theorem}
\begin{proof}
As before, we construct smooth approximations, $\g_{\eps}$, that converge to $\g$ in $C^0(M) \cap W^{2, n/2}_{\loc}(M)$ as $\eps \to 0$. The only difference now is that, as $\eps \to 0$, we have
\[
\norm{\sepsmin}_{L^{n/2}(M, \g_{\eps})} \to \norm{\sgmin}_{L^{n/2}(M, \g)} \neq 0.
\]
Since $c_1[\g] \, \norm{\sgmin}_{L^{n/2}(M, \g)} < a_n$, it follows that $c_1[\g_{\eps}] \, \norm{\sepsmin}_{L^{n/2}(M, \g_{\eps})} < a_n$ for sufficiently small $\eps > 0$. Moreover, since $\g_{\eps}$ agree with $\g$ away from a compact set, the functions $\sepsmin$ have compact support. For such $\eps$, we construct the conformally transformed metrics $\ghat_{\eps}$ as before. These (smooth) metrics have non-negative scalar curvature and are asymptotically flat, therefore $\adm(\g_{\eps}) \ge 0$. We therefore have
\begin{align}
\adm(\g_{\eps}) &\ge - 2 A_{\eps}
\nonumber
\\
&= \frac{2}{(n-2) \omega_{n-1}} \int_M \left[ |\nabla_{\g} u_{\eps}|_{\g}^2 - \frac{1}{a_n} \sepsmin u_{\eps}^2 \right] \, d\mu_{\g_\eps}
\nonumber
\\
&\ge - \frac{2}{(n-2) \omega_{n-1}} \frac{1}{a_n} \norm{ \sepsmin}_{L^{n/2}(M, \g_{\eps})} \norm{u_{\eps}}_{L^{n^*} \left(\supp \sepsmin, \g_{\eps} \right)}^2.
\label{negmass}
\end{align}
Since $u_{\eps} = 1 + v_{\eps}$, the estimate~\eqref{vbound1} implies that
\begin{align*}
\norm{u_{\eps}}_{L^{n^*}(\supp \sepsmin, \g_{\eps})}
&\le \mu_{\g_{\eps}}(\supp \sepsmin)^{1/n^*} + \frac{c_1[\g_{\eps}] \, \norm{\sepsmin}_{L^{n/2}(M, \g_{\eps})}}{a_n - c_1[\g_{\eps}] \, \norm{\sepsmin}_{L^{n/2}(M, \g_{\eps})}} \mu_{\g_{\eps}}(\supp \sepsmin)^{1/n^*}
\\
&= \frac{1}{1 - \frac{1}{a_n} c_1[\g_{\eps}] \, \norm{\sepsmin}_{L^{n/2}(M, \g_{\eps})}} \mu_{\g_{\eps}}(\supp \sepsmin)^{1/n^*}.
\end{align*}
Since $\adm(\g_{\eps}) = \adm(\g)$, we therefore deduce from~\eqref{negmass} that
\[
\adm(\g) \ge - \frac{1}{2 (n-1) \omega_{n-1}}
\, \frac{\norm{ \sepsmin}_{L^{n/2}(M, \g_{\eps})}}{(1 - \frac{1}{a_n} c_1[\g_{\eps}] \, \norm{\sepsmin}_{L^{n/2}(M, \g_{\eps})})^2}
\, \mu_{\g_{\eps}}(\supp \sepsmin)^{2/n^*}
\]
for all $\eps > 0$. Taking the limit as $\eps \to 0$ gives~\eqref{lowerbdonmass}.
\end{proof}

\section{Approximation by smooth metrics}
\label{sec:approximation}

In this final section, we show how the techniques that we have developed may be used to construct smooth metrics with non-negative scalar curvature that approximate rough metrics with non-negative scalar curvature in the distributional sense. More precisely, for example, we will show that, given a metric $\g \in W^{2, p}_{\loc}(\Omega)$, where $p > \frac{n}{2}$ defined on an open set $\Omega$ that has non-negative scalar curvature in the distributional sense, one can construct smooth metrics $\ghat_{\eps}$ with non-negative scalar curvature on slightly smaller open sets with the property that $\ghat_{\eps}$ converges locally uniformly to $\g$ as $\eps \to 0$. We also show, however, that our approach breaks down when $\g$ is only assumed to have regularity $C^0(\Omega) \cap W^{2, n/2}_{\loc}(\Omega)$. In particular, for $\g \in C^0(\Omega) \cap W^{2, n/2}_{\loc}(\Omega)$, the smooth metrics $\ghat_{\eps}$ that we shall construct do not appear to converge to $\g$ in $L^{\infty}_{\loc}(\Omega)$, although they do converge to $\g$ in $L^p_{\loc}(\Omega)$ for all $p < \infty$ (in a sense that will be clarified).

As such, let $\Omega$ be an open, smooth manifold without boundary and $\g$ a Riemannian metric on $\Omega$ of regularity $W^{2, p}_{\loc}(\Omega)$ for some $\frac{n}{2} < p < \infty$. (Note that the Sobolev embedding theorem implies that $\g$ is continuous.) An argument similar to the proof of Proposition~\ref{sinLn/2} shows that the scalar curvature of $\g$ is a well-defined element of $L^p_{\loc}(\Omega)$. We assume that $\g$ has non-negative scalar curvature in the distributional sense. By taking the convolution of the components of $\g$ with a mollifier in local charts, we construct a family of smooth symmetric $(0, 2)$ tensor fields $\g_{\eps}$ on $\Omega$ with components $g_{\eps, ij} := g_{ij} * \rho_{\eps}$. As $\eps \to 0$, the metrics $\g_{\eps}$ converge to $\g$ uniformly on compact subsets of $\Omega$, and in $W^{2, p}_{\loc}(\Omega)$. Given any compact set $K \subset \Omega$, for sufficiently small $\eps > 0$, the $\g_{\eps}$ are smooth Riemannian metrics on $K$.

Let $\Omega_{\eps}$ be open subsets of $\Omega$ with the property that the sets $K_{\eps} := \overline{\Omega_{\eps}}$ are a compact exhaustion of $\Omega$ and that $\partial\Omega_{\eps}$ are smooth. Relabelling the sets $K_{\eps}$, if necessary, we assume that $\g_{\eps}$ is a smooth Riemannian metric on an open neighbourhood of $K_{\eps}$. We now solve the Dirichlet problem
\[
\Delta_{\g_{\eps}} u_{\eps} + \frac{1}{a_n} \sepsmin u_{\eps} = 0, \qquad \left. u_{\eps} \right| \delta \Omega_{\eps} = 1
\]
to give (smooth) conformally transformed metrics
\bel{ghat}
\ghat_{\eps} := u_{\eps}^{\frac{4}{n-2}} \g_{\eps}
\ee
with non-negative scalar curvature on the set $\Omega_{\eps}$. Our aim is to derive global estimates for $u_{\eps}$ on the set $\Omega$, and to show that the metrics $\ghat_{\eps}$ converge locally uniformly to $\g$.

\begin{remark}
Although the metrics $\g_{\eps}$ are smooth, the only quantity that we will be able to control when they arise from smoothing the metric $\g$ is the Sobolev constant $c_1[\g_{\eps}]$ and the norm $\norm{\sepsmin}_{L^p(\g_{\eps})}$. As such, all of our estimates for the $u_{\eps}$ should only involve these quantities.
\end{remark}

\subsection{Global estimates for the conformal factor}
\label{sec:global}

Again, for simplicity, we first develop our results for an arbitrary smooth metric $\h$ on an open set $\Omega$ with smooth boundary $\delta\Omega$. For the duration of this section, $L^p$ norms will all be with respect to $\h$, and denoted by subscripts (e.g. $\norm{f}_p$ will denote $\norm{f}_{L^p(\Omega, \h)}$, etc). For brevity, we denote $\mu_{\h}(\Omega)$ by $|\Omega|$. The Sobolev constant will be denoted by simply $c_1$.

Letting the desired conformal factor be $u = 1 + v$, we require $v$ to be a solution of the Dirichlet problem
\bel{DirProblem}
\Delta_{\h} v + f v + f = 0 \quad \mbox{in $\Omega$},
\qquad
\left. v \right| \partial \Omega = 0,
\ee
where $f := \frac{1}{a_n} \shmin \ge 0$.

\

The main results of this section are lower and upper bounds for the solution $v$ of~\eqref{DirProblem} on $\Omega$. We begin with the lower bound.

\begin{theorem}
\label{thm:lowervest}
If $\norm{f}_p$ is sufficiently small that
\bel{psmall}
c_1 \, \norm{f}_p |\Omega|^{\frac{2}{n}-\frac{1}{p}} < 1,
\ee
then, on $\Omega$, we have
\bel{lowerMoserbd}
v \ge - \chi^{\tau} \frac{\left[ c_1 \norm{f}_p \, |\Omega|^{\frac{2}{n}-\frac{1}{p}} \right]^{\frac{\sigma}{2}+1}}{1-c_1 \, \norm{f}_p \, |\Omega|^{\frac{2}{n}-\frac{1}{p}}},
\ee
where $\chi := \frac{n^*}{(2p)^*}, \sigma = \frac{1}{\chi - 1}, \tau = \frac{\chi}{\left( \chi - 1 \right)^2}$.
\end{theorem}

The proof of this Theorem requires the following preliminary result, the proof of which is a straightforward modification of the calculation that leads to the bounds in Proposition~\ref{prop:bounds}.

\begin{lemma}
If $\norm{f}_{L^p}$ is sufficiently small that the condition~\eqref{psmall} holds, then the solution of~\eqref{DirProblem} satisfies the inequality
\bel{pvbound2}
\norm{v}_{n^*}
\le \frac{c_1 \, \norm{f}_p |\Omega|^{\frac{2}{n}-\frac{1}{p}}}{1 - c_1 \, \norm{f}_p |\Omega|^{\frac{2}{n}-\frac{1}{p}}} |\Omega|^{\frac{1}{n^*}}.
\ee
\end{lemma}

\noindent{\textbf{Notation}}: For brevity, we let $A_p := c_1 \, \norm{f}_p |\Omega|^{\frac{2}{n}-\frac{1}{p}}$, which is assumed less that $1$.

\begin{proof}[Proof of Theorem~\ref{thm:lowervest}]
Our proof follows, to some extent, the Moser iteration proof of Theorem~8.15 of~\cite{GT}. However, we need to keep explicit control over the constants that appear in the estimates.

In order to get a lower bound on $v$, we set $\bar{v}=-v$, and note that $\bar{v}$ satisfies the equation
\[
\Delta_{\h} \bar{v} + f \bar{v} = f \quad \mbox{in $\Omega$},
\qquad
\left. \bar{v} \right| \partial \Omega = 0.
\]
The weak form of our differential equation is now given by
\[
\int_{\Omega} \la \nabla \Phi, \nabla \bar{v} \ra_{\h} \, d\mu_{\h} = \int_{\Omega} f \left( \bar{v} - 1 \right) \Phi \, d\mu_{\h},
\]
where $\Phi$ is a test function, which we may take to lie in $H_0^1(\Omega)$. We take $\Phi \ge 0$ in which case, using the fact that $f \ge 0$, we obtain
\bel{weak}
\int_{\Omega} \la \nabla \Phi, \nabla \bar{v} \ra_{\h} \, d\mu_{\h} \le \int_{\Omega} f \bar{v} \Phi \, d\mu_{\h}.
\ee
Let $w := {\bar{v}}_+$. Note that $dw = 0$ almost everywhere on the set $\bar{v} \le 0$ and $dw = d\bar{v}$ on the set $\bar{v} > 0$. Let $G \colon [0, \infty) \to \R$ be a non-decreasing Lipschitz function with $G(0) = 0$. It follows~\cite[Theorem~3.1.7]{Morrey} that $\Phi(x) := G(w(x))$ lies in $H_0^1(\Omega)$ and is therefore an admissible test function that we can insert in~\eqref{weak}. Moreover, $\Phi = 0$ on the set $\bar{v} \le 0$. We therefore find that
\[
\int_{\Omega} G'[w] |\nabla w|_{\h}^2 \, d\mu_{\h} \le \int_{\Omega} f w G[w] \, d\mu_{\h}.
\]
Let $H \colon [0, \infty) \to [0, \infty)$ be a function defined by
\[
G'(t) = \left( H'(t) \right)^2, \qquad H(0) = 0.
\]
We then have
\[
\int_{\Omega} |\nabla H[w]|_{\h}^2 \, d\mu_{\h} \le \int_{\Omega} f w G[w] \, d\mu_{\h}.
\]

Given $\beta \ge 1$ and $N > 0$, we take $H$ to be the Lipschitz function
\bel{H}
H[t] := \begin{cases} t^{\beta} &0 \le t \le N, \\ \beta N^{\beta - 1} \left( t - N \right) &t > N.\end{cases}
\ee
In this case, $G[t] \le t G'[t] = t H'[t]^2$ for $t \ge 0$. Since $H[w] \in H_0^1(\Omega)$, the Sobolev inequality implies that
\[
\norm{H[w]}^2
\le c_1 \int_{\Omega} |\nabla H[w]|_{\h}^2 \, d\mu_{\h}
\le c_1 \int_{\Omega} f \left( w H'[w] \right)^2 \, d\mu_{\h}
\]
Taking the limit as $N \to \infty$, we deduce that
\[
\norm{w}_{\beta n^*}^2
\le \beta^2 c_1 \int_{\Omega} f w^{2\beta} \, d\mu_{\h}
\le \beta^2 c_1 \norm{f}_p \norm{w^{2\beta}}_{\frac{p}{p-1}}.
\]
Letting
\[
C := \left[ c_1 \norm{f}_p \right]^{1/2},
\]
we therefore have
\[
\norm{w}_{\beta n^*} \le \left( C \beta \right)^{1/\beta} \norm{w}_{\beta (2p)^*}.
\]
Defining
\[
\chi := \frac{n^*}{(2p)^*} > 1,
\]
we rewrite this inequality in the form
\[
\norm{w}_{\beta \chi (2p)^*} \le \left( C \beta \right)^{1/\beta} \norm{w}_{\beta (2p)^*}
\]
Starting with $\beta = \chi^m$, with $m$ a positive integer, and iterating, we have
\begin{align*}
\norm{w}_{\chi^{m+1} (2p)^*}
\le \left( C \chi^m \right)^{1/\chi^m} \norm{w}_{\chi^m (2p)^*}
\le \dots \le
C^{\sigma_m} \chi^{\tau_m} \norm{w}_{n^*},
\end{align*}
where
\[
\sigma_m = \sum_{i=1}^m \frac{1}{\chi^i}, \qquad
\tau_m = \sum_{i=1}^m \frac{i}{\chi^i}.
\]
Letting $m \to \infty$, we deduce that
\[
\sup_{\Omega} w \le C^{\sigma} \chi^{\tau} \norm{w}_{n^*},
\]
where
\[
\sigma = \frac{1}{\chi - 1}, \qquad
\tau = \frac{\chi}{\left( \chi - 1 \right)^2}.
\]
We now use the fact that $\sup \bar{v} \le \sup \bar{v}_+ = \sup w$ and
\[
\norm{w}_{n^*} = \norm{\bar{v}_+}_{n^*} = \norm{v_-}_{n^*} \le \norm{v}_{n^*}
\]
to give
\[
\sup_{\Omega} \bar{v}
\le C^{\sigma} \chi^{\tau} \frac{A_p}{1-A_p} |\Omega|^{1/n^*}
\equiv A_p^{\sigma/2} \chi^{\tau} \frac{A_p}{1-A_p},
\]
where we recall that we have defined $A_p := c_1 \, \norm{f}_{L^p} |\Omega|^{\frac{2}{n}-\frac{1}{p}}$, and we have substituted the inequality~\eqref{pvbound2} in the final step. We have thus established~\eqref{lowerMoserbd}.
\end{proof}

Acquiring an upper estimate for $v$ is a little more delicate:

\begin{theorem}
\label{thm:uppervest}
If $\norm{f}_p$ is sufficiently small that the condition~\eqref{psmall} holds, then we have
\bel{vupperbd}
v \le \chi^{\tau} \frac{c_1 \, \norm{f}_p |\Omega|^{\frac{2}{n}-\frac{1}{p}}}{1 - c_1 \, \norm{f}_p |\Omega|^{\frac{2}{n}-\frac{1}{p}}},
\ee
where $\chi := \frac{n^*}{(2p)^*}$, $\tau = \frac{\chi}{\left( \chi - 1 \right)^2}$.
\end{theorem}
\begin{proof}
Let $w := v_+ \in H_0^1(\Omega)$. Given $\beta \ge 1$ and $N > 0$, it follows that $H[w]$ as defined in~\eqref{H} lies in $H_0^1(\Omega)$. The Sobolev inequality and integration by parts then yields
\begin{align}
\norm{w^{\beta}}_{n^*}^2
&\le c_1 \int_{\Omega} |\nabla w^{\beta}|^2 \, d\mu_{\h}
\nonumber
\\
&= - c_1 \frac{\beta^2}{2\beta - 1} \int_{\Omega} w^{2 \beta - 1} \Delta w \, d\mu_{\h}
\nonumber
\\
&= c_1 \frac{\beta^2}{2\beta - 1} \int_{\Omega} w^{2 \beta - 1} f \left( 1 + w \right) \, d\mu_{\h}
\nonumber
\\
&\le c_1 \beta^2 \norm{f}_{L^p(\Omega)} \left[ \norm{w^{2 \beta - 1}}_{p/(p-1)} + \norm{w^{2 \beta}}_{p/(p-1)} \right].
\label{MoserLp}\end{align}
In addition, {\H}'s inequality yields
\[
\norm{w^{2 \beta - 1}}_{p/(p-1)}
= \norm{w}_{(2 \beta - 1) \frac{p}{p-1}}^{2 \beta - 1}
\le \norm{w}_{2 \beta \frac{p}{p-1}}^{2 \beta - 1} |\Omega|^{\frac{1}{2\beta} \left( 1 - \frac{1}{p} \right)}
= \norm{w}_{\beta (2p)^*}^{2 \beta - 1} |\Omega|^{\frac{1}{\beta (2p)^*}}.
\]
We therefore have
\bel{preclaim}
\norm{w}_{\beta n^*}^{2\beta}
\le c_1 \beta^2 \norm{f}_p
\norm{w}_{\beta (2p)^*}^{2 \beta - 1}
\left[ |\Omega|^{\frac{1}{\beta (2p)^*}} + \norm{w}_{\beta (2p)^*} \right]
\ee

Letting $\chi := \frac{n^*}{(2p)^*} > 1$, we now claim that it follows from~\eqref{preclaim} and~\eqref{pvbound2} that \bel{claim}
\norm{w}_{\chi^m n^*} \le \frac{A_p}{1-A_p} \chi^{\frac{1}{\chi} + \dots + \frac{m}{\chi^m}} |\Omega|^{\frac{1}{\chi^m n^*}}
\ee
for all integers $m \ge 1$. We prove this claim by induction. From~\eqref{preclaim}, taking $\beta = \chi$, we have
\[
\norm{w}_{\chi n^*}^{2\chi}
\le c_1 \chi^2 \norm{f}_p
\norm{w}_{n^*}^{2 \chi - 1}
\left[ |\Omega|^{\frac{1}{n^*}} + \norm{w}_{L^n*(\Omega)} \right]
\]
We now note that
\[
\norm{w}_{n^*} =
\norm{v_+}_{n^*} \le
\norm{v}_{n^*}
\le \frac{A_p}{1 - A_p} |\Omega|^{\frac{1}{n^*}},
\]
where the final inequality follows from~\eqref{pvbound2}. We therefore have
\begin{align*}
\norm{w}_{\chi n^*}^{2\chi}
&\le c_1 \chi^2 \norm{f}_p
\left( \frac{A_p}{1 - A_p} |\Omega|^{\frac{1}{n^*}} \right)^{2 \chi - 1}
\left[ |\Omega|^{\frac{1}{n^*}} + \frac{A_p}{1 - A_p} |\Omega|^{\frac{1}{n^*}} \right]
\\
&= c_1 \chi^2 \norm{f}_p |\Omega|^{\frac{2\chi}{n^*}}
\frac{A_p^{2 \chi - 1}}{(1 - A_p)^{2 \chi}}
\\
&= \chi^2 |\Omega|^{\frac{2}{(2p)^*}-\frac{2}{n}+\frac{1}{p}}
\left( \frac{A_p}{1 - A_p} \right)^{2 \chi}
\\
&= \chi^2 |\Omega|^{\frac{1}{\chi n^*}}
\left( \frac{A_p}{1 - A_p} \right)^{2 \chi}
\end{align*}
Therefore, the claim is established for $m = 1$. Assuming it to be true for some $m \ge 1$, taking $\beta = \chi^{m+1}$ in~\eqref{preclaim}, we have
\begin{align*}
\norm{w}_{\chi^{m+1} n^*}^{2\chi^{m+1}}
&\le c_1 \chi^{2m+2} \norm{f}_p
\norm{w}_{\chi^m n^*}^{2 \chi^{m+1} - 1}
\left[ |\Omega|^{\frac{1}{\chi^m n^*}} + \norm{w}_{\chi^m n*} \right]
\\
&\le
c_1 \chi^{2m+2} \norm{f}_p
\left[ \frac{A_p}{1-A_p} \chi^{\frac{1}{\chi} + \dots + \frac{m}{\chi^m}} |\Omega|^{\frac{1}{\chi^m n^*}} \right]^{2 \chi^{m+1} - 1}
\\
&\hskip 1cm \times \left[ |\Omega|^{\frac{1}{\chi^m n^*}} + \frac{A_p}{1-A_p} \chi^{\frac{1}{\chi} + \dots + \frac{m}{\chi^m}} |\Omega|^{\frac{1}{\chi^m n^*}} \right]
\\
&\le
c_1 \chi^{2m+2} \norm{f}_p
\left[ \frac{A_p}{1-A_p} \chi^{\frac{1}{\chi} + \dots + \frac{m}{\chi^m}} |\Omega|^{\frac{1}{\chi^m n^*}} \right]^{2 \chi^{m+1} - 1}
\\
&\hskip 1cm \times |\Omega|^{\frac{1}{\chi^m n^*}}
\chi^{\frac{1}{\chi} + \dots + \frac{m}{\chi^m}} \left[ 1 + \frac{A_p}{1-A_p} \right]
\\
&=
c_1 \chi^{2m+2} \norm{f}_p
\frac{A_p^{2 \chi^{m+1} - 1}}{(1-A_p)^{2 \chi^{m+1}}}
|\Omega|^{\frac{2 \chi}{n^*}}
\chi^{2 \chi^{m+1} \left( \frac{1}{\chi} + \dots + \frac{m}{\chi^m} \right)},
\\
&=
\left( \frac{A_p}{1-A_p} \right)^{2 \chi^{m+1}}
|\Omega|^{\frac{2}{n^*}}
\chi^{2 \chi^{m+1} \left( \frac{1}{\chi} + \dots + \frac{m}{\chi^m} + \frac{m+1}{\chi^{m+1}}\right)},
\end{align*}
where we have used the fact that $1 \le \chi$ in the third inequality. Taking the $2 \chi^{m+1}$ root, we have therefore established the claim~\eqref{claim} for $m+1$. Therefore, by induction,~\eqref{claim} holds for all $m \ge 1$.

Taking the limit as $m \to \infty$ in~\eqref{claim} now yields the inequality~\eqref{vupperbd}.
\end{proof}

\subsection{Approximations to $\g$}

Applying Theorems~\ref{thm:lowervest} and~\ref{thm:uppervest} to the metrics $\g_{\eps}$, we have the following result.

\begin{proposition}
Let $\g \in W^{2, p}_{\loc}(\Omega)$ with $p > \frac{n}{2}$ have non-negative scalar curvature in the distributional sense, and let $\g_{\eps}$ be smooth approximating metrics as above. Given any compact subset $K \subset \Omega$, let $\eps(K) > 0$ be such that $K \subset \Omega_{\eps}$ for all $\eps < \eps(K)$. Then the solutions $\left. v_{\eps} \right| K$ for $\eps < \eps(K)$, converge uniformly to zero on $K$ as $\eps \to 0$.
\end{proposition}
\begin{proof}
Since $\norm{\sepsmin}_{L^p(\Omega, \g_{\eps})} \to 0$ as $\eps \to 0$, we may assume, without loss of generality, that the condition $c_1[\g_{\eps}] \, \norm{\sepsmin}_{L^p(\Omega, \g_{\eps})} \mu_{\g_{\eps}}(\Omega_{\eps})^{\frac{2}{n}-\frac{1}{p}} < a_n$ is satisfied for all $\eps < \eps(K)$. The uniform bounds on $v_{\eps}$ on $\Omega_{\eps}$ given by~\eqref{lowerMoserbd} and~\eqref{vupperbd} both converge to zero as $\eps \to 0$, therefore implying that $\left. v_{\eps} \right| K$ converge uniformly to zero on $K$.
\end{proof}

We therefore have the following result regarding approximations of $\g$.

\begin{theorem}
Let $\g$ be a Riemannian metric on an open set $\Omega$ of regularity $W^{2, p}_{\loc}(\Omega)$, $p > \frac{n}{2}$ with non-negative scalar curvature in the distributional sense. Then there exists a family of smooth, Riemannian metrics $\{ \ghat_{\eps} \mid \eps > 0 \}$ defined on open sets $\Omega_{\eps} \subset \Omega$ such that $K_{\eps} := \overline{\Omega_{\eps}}$ are a compact exhaustion of $\Omega$, such that:
\begin{itemize}
\item[$\bullet$] The $\ghat_{\eps}$ are have non-negative scalar curvature;
\item[$\bullet$] $\ghat_{\eps}$ converge locally uniformly to $\g$ as $\eps \to 0$.
\end{itemize}
\end{theorem}

The use of the sets $\Omega_{\eps}$ was necessary since, by smoothing $\g$ on an open set $\Omega$, we will only get locally uniform convergence of $\g_{\eps}$ to $\g$. We may remove the use of this construction if $\Omega$ is contained in a larger open set.

\begin{theorem}
Let $\g$ be a Riemannian metric on an open set $\Omega'$ of regularity $W^{2, p}_{\loc}(\Omega')$, $p > \frac{n}{2}$, that has non-negative scalar curvature in the distributional sense. Let $\Omega$ be an open subset of $\Omega'$ with compact closure $K \subset \Omega'$, smooth boundary $\partial\Omega$, such that $\Omega'$ is an open neighbourhood of $K$. Then there exists a family of smooth, Riemannian metrics $\{ \ghat_{\eps} \mid \eps > 0 \}$ on $K$ with the following properties:
\begin{itemize}
\item[$\bullet$] The $\g_{\eps}$ are have non-negative scalar curvature;
\item[$\bullet$] $\ghat_{\eps}$ converge uniformly to $\g$ on $K$ as $\eps \to 0$.
\end{itemize}
\end{theorem}
\begin{proof}
Smoothing $\g$ by convolution in charts on $\Omega'$ gives, for all sufficiently small $\eps$, a family of smooth Riemannian metrics $\g_{\eps}$ on the set $K$. We now solve the Dirichlet problem
\[
\Delta_{\g_{\eps}} u_{\eps} + \frac{1}{a_n} \sepsmin u_{\eps} = 0, \qquad \left. u_{\eps} \right| \partial\Omega = 1.
\]
The conformally transformed metrics $\ghat_{\eps}$ are then metrics with non-negative scalar curvature on the set $K$. The bounds on $v_{\eps}$ given by~\eqref{lowerMoserbd} and~\eqref{vupperbd} now hold on the set $\Omega$, implying that $\left. u_{\eps} \right| K$ converge uniformly to $1$ on $K$. Since the $\g_{\eps}$ converge uniformly to $\g$ on $K$, it follows that the $\ghat_{\eps}$ converge uniformly to $\g$ on $K$.
\end{proof}

Finally, in the context of the positive mass theorem, our results give the following.

\begin{theorem}
Let $M$ be a smooth manifold and $\g$ an asymptotically flat, Riemannian metric on $M$ of regularity $W^{2, p}_{\loc}(M)$ and smooth outside of the compact set $K$. Then there exist smooth metrics $\ghat_{\eps}$ on $M$ with non-negative scalar curvature that converge locally uniformly to $\g$ as $\eps \to 0$. In particular, $\g$ can be approximated locally uniformly by smooth metrics with non-negative ADM mass.
\end{theorem}
\begin{proof}
The only non-trivial point is to note that in the estimates~\eqref{lowerMoserbd} and~\eqref{vupperbd}, the occurrences of the set $\Omega$ may be replaced by $\Omega \cap \supp f$. Since $\sepsmin$ will have compact support, these factors will remain finite, so the $v_{\eps}$ will still converge to zero.
\end{proof}

\begin{remark}
In the region where the metric $\g$ is smooth, the metrics $\g_{\eps}$ will converge to $\g$ in $C^{\infty}$.
\end{remark}

\subsection{Breakdown of Moser iteration for $p=n/2$}

Finally, we briefly consider the case where the metric $\g$ is of regularity $C^0(\Omega) \cap W^{2, n/2}_{\loc}(\Omega)$. In this case, the Moser iteration arguments used in Section~\ref{sec:global} to derive an $L^{\infty}$ bound on solutions of the Dirichlet problem break down. In particular, taking $p = \frac{n}{2}$ in the inequality~\eqref{MoserLp}, we deduce that
\begin{align*}
\norm{w}_{L^{\beta n^*}(\Omega)}^{2\beta}
&\le c_1 \beta^2 \norm{f}_{L^{n/2}(\Omega)} \left[ \norm{w}_{L^{(\beta-\frac{1}{2}) n^*}(\Omega)}^{2\beta-1} + \norm{w}_{L^{n^*}(\Omega)}^{2\beta} \right].
\end{align*}
Assuming that $c_1 \beta^2 \norm{f}_{L^{n/2}(\Omega)} < 1$, this inequality yields a bound on $\norm{w}_{L^{\beta n^*}(\Omega)}$ in terms of $\norm{w}_{L^{(\beta-\frac{1}{2}) n^*}(\Omega)}$. However, it is clear that such an iteration process breaks down when $\beta$ is of the order $\frac{1}{\sqrt{c_1 \norm{f}_{L^{n/2}(\Omega)}}}$. As such, the Moser iteration argument breaks down after a finite number of repetitions.%
\footnote{See, for instance,~\cite{RS} and~\cite[Thm.~4.4]{HanLin} for examples of this phenomenon.}
This means that we can establish an $L^p$ bound on $v$, where
\[
p \sim \beta n^* \sim n^* / \sqrt{c_1 \norm{f}_{L^{n/2}(\Omega)}} < \infty.
\]
Nevertheless, we note that $p(\eps) \to \infty$ as $\eps \to 0$.

These observations imply the following result:

\begin{theorem}
Let $\Omega'$ be an open set, $\g$ a continuous Riemannian metric on $\Omega'$ of regularity $W^{2, n/2}_{\loc}(\Omega')$ with non-negative scalar curvature in the distributional sense. Let $\Omega$ be an open subset of $\Omega'$ with compact closure and $\partial\Omega$ smooth. Then, there exists a family of smooth metrics $\ghat_{\eps}$ on $\Omega$ with non-negative scalar curvature such that $\ghat_{\eps} \to \g$ in $L^p_{\loc}(\Omega)$ as $\eps \to 0$, for all $p < \infty$.
\end{theorem}

\begin{remark}
It appears in general that the $\ghat_{\eps}$ do not converge locally uniformly to $\g$. In this regard, we note that it is possible to construct functions $f_{\eps} \ge 0$ on the unit ball in $\R^2$ with the property that $f_{\eps} \to 0$ in $L^1$ as $\eps \to 0$, but the corresponding solutions of the Dirichlet problem%
\footnote{$\Delta_0$ denotes the Laplacian $\partial_x^2 + \partial_y^2$.}
\[
\Delta_0 u_{\eps} + f_{\eps} u_{\eps} = 0 \mbox{ in $B(0, 1)$}, \qquad u_{\eps} = 1 \mbox{ on $\partial B(0, 1)$}
\]
have the property that $u_{\eps}(0) \to \infty$ as $\eps \to 0$.\footnote{We are grateful to Dr.\ Jonathan Bevan for pointing this out to us.}.
\end{remark}

It therefore appears likely that metrics in $C^0 \cap W^{2, n/2}_{\loc}$ with non-negative scalar curvature in the distributional sense cannot, in general, be approximated in $C^0$ by smooth metrics with non-negative scalar curvature. Similarly, when considering limits of sequences of smooth metrics with non-negative scalar curvature, one might expect to encounter non-compactness or ``bubbling off'' phenomena. Based on the extremal case of the Sobolev embedding theorem, however, we would expect that such metrics can be approximated in BMO or appropriate Orlicz spaces.

\section{A remark on rigidity}
An interesting point to note in our work is that we are unable to achieve a rigidity result with our approach. The rigidity part of the classical positive mass theorem asserts that if a metric is asymptotically flat, has non-negative scalar curvature and the ADM mass is zero, then $(M, \g)$ is isometric to $\R^n$ with the Euclidean metric. Our methods are insufficient to prove such a statement for our class of metrics. The proof of rigidity in Schoen and Yau~\cite{SY} involves a perturbation of the metric $\g$ by its Ricci tensor. This approach cannot be adapted to our metrics, since the Ricci tensor of the metric $\g$ lies in $L^{n/2}_{\loc}(M)$, so such a perturbation of $\g$ would not preserve the required regularity of the metric. The proof of the rigidity part of the positive mass theorem in the paper of Miao~\cite{Miao} in itself requires more regularity of the metric, in order to apply the results of~\cite{BrayFinster} (which, in turn, use Witten's technique, and therefore works only for spin manifolds). Finally, in the Ricci-flow approach of~\cite{McS} (perhaps the most promising approach to proving rigidity) it is required that one has an $L^{\infty}$ bound on the negative part of the scalar curvature of the smooth approximating metrics in order to show that the solution of the $h$-flow equation starting with the metric $\g$ has non-negative scalar curvature. In our case, we have only an $L^{n/2}$ bound on the negative part of the scalar curvature of the approximating metrics, which is insufficient. It therefore appears that all known ways to prove the rigidity result for non-spin manifolds break down (by some distance) for our class of metrics.

For the case of spin manifolds, the situation is rather better (cf.~\cite{Bartnik, BCPreprint, LL}). In Witten's argument, the rigidity part of the positive mass theorem follows directly from the fact that, if the mass of a metric is zero, then one has a basis of parallel spinor fields. As such, the spin-connection is flat, thus the full curvature tensor of $\g$ is flat. A theorem of Cartan then implies that the manifold is isometric to $\R^{n-k} \times T^k$, where $T^k$ is a flat torus of dimension $k$.%
\footnote{This is a smooth result. It is not completely clear to us what regularity conditions a Riemannian metric must satisfy in order vanishing of its Riemann curvature implies the existence of a locally isometry with Euclidean space.}
Asymptotic flatness then implies that $k = 0$. As such, the rigidity part of the positive mass theorem for spin manifolds appears to require no additional regularity.

Quite why there is such a mismatch between the regularity required to prove the positive mass theorem on spin manifolds and non-spin manifolds is, at this point, rather unclear.

\end{document}